\documentclass[amssymb,amsfonts,refcheck,12pt,verbatim,righttag]{amsart}
\usepackage{amssymb,amscd}
\usepackage{color}
\numberwithin{equation}{section} %

\def\colon{{:}\;}

\def\R {\Bbb R}

\newcommand{\beq}{\begin{equation}}
\newcommand{\eeq}{\end{equation}}

\newcommand{\ben}{\begin{eqnarray}}
\newcommand{\een}{\end{eqnarray}}

\newcommand{\bet}{\begin{eqnarray*}}
\newcommand{\eet}{\end{eqnarray*}}

\newtheorem{thm}{Theorem}[section]
\newtheorem{que}[thm]{Question}
\newtheorem{lem}[thm]{Lemma}
\newtheorem{pro}[thm]{Proposition}
\newtheorem{cor}[thm]{Corollary}
\newtheorem{ex}[thm]{Example}
\newtheorem{de}[thm]{Definition}
\newtheorem{re}[thm]{Remark}

\def\R {\Bbb R}
\def\N {\Bbb N}
\def\M {{\mathcal M}}
\def\W {{\mathcal W}}

\def\A{{\mathcal A}}
\def\B {{\mathcal B}}
\def\E {{\mathcal E}}
\def\a{{\bf a}}
\def\b{{\bf b}}

\def\I{{\mathcal I}}
\def\J{{\mathcal J}}
\def\Z{{\Bbb Z}}

\def\Q{{\mathcal Q}}

\def\ba{{\bf a}}
\def\D{{\mathcal D}}
\def\C{{\mathcal C}}
\def\T{{\Bbb T}}
\def\htop{h_{\rm top}}
\theoremstyle{plain}

\begin{document}
\baselineskip 15pt

\title
{Variational principle for weighted topological pressure}%Weighted thermodynamic formalism for topological dynamical systems}

\author{De-Jun FENG}
\address{
Department of Mathematics\\
The Chinese University of Hong Kong\\
Shatin,  Hong Kong\\
}
\email{djfeng@math.cuhk.edu.hk}

\author{Wen Huang}
\address{Department of Mathematics, Sichuan University,
Chengdu, Sichuan 610064, China \\ {\it  and }\\
School of Mathematical Sciences, University of Science and Technology of
China, Hefei, Anhui 230026, China\\
}
\email{wenh@mail.ustc.edu.cn }
% \thanks {The first author
%was partially supported by the direct grant and RGC grants in the Hong Kong Special Administrative
%Region, China (projects CUHK400706, CUHK401008).  The second author was partially supported by NSFC
%(Grant 11225105, 11371339, 11431012).}
\date{}

\begin{abstract}
Let $\pi:X\to Y$ be a factor map, where $(X,T)$ and  $(Y,S)$ are topological dynamical systems.    Let $\ba=(a_1,a_2)\in \R^2$ with $a_1>0$ and $a_2\geq 0$, and $f\in C(X)$.  The $\ba$-weighted topological pressure of $f$, denoted by $P^\ba(X, f)$,  is defined by resembling the Hausdorff dimension of subsets of self-affine carpets.   We prove the following variational principle:
$$
P^\ba(X, f)=\sup\left\{a_1h_\mu(T)+a_2h_{\mu\circ\pi^{-1}}(S)+\int f \;d\mu\right\},
$$
where the supremum is taken over the $T$-invariant measures on $X$.  It not only generalizes the variational principle of classical topological pressure, but also provides a  topological extension of  dimension theory of  invariant sets and measures on the torus under affine diagonal endomorphisms. A higher dimensional version of the result is also established.
\end{abstract}

\maketitle
\setcounter{section}{0}
\section{Introduction}
\setcounter{equation}{0}

Inspired by the theory of Gibbs states in statistical mechanics, Ruelle  \cite{Rue73} introduced the notion of topological pressure  to the theory of dynamical systems and established a variational principle for it.   Ruelle only considered the case when the underlying dynamical systems satisfy expansiveness and specification. Later Walters \cite{Wal75} generalized these results to general topological dynamical systems.  Topological pressure, and the associated variational principle and equilibrium measures constitute the main components of the thermodynamic formalism \cite{Rue78}. They play important roles in  dimension theory of dynamical systems. Indeed they provide as a basic tool in studying dimension of invariant sets and measures for conformal dynamical systems (see e.g. \cite{Bow79, Rue82, Pes97}).

In this paper we aim to introduce a generalized notion of pressure for factor maps between general topological dynamical systems, and establish a variational principle for it.  To be more precise, let us introduce some notation first. We say that $(X, T)$ is a {\it topological dynamical system} (TDS) if $X$ is a compact metric space and $T$ is a continuous map from $X$ to $X$. Let $(X,T)$ and  $(Y,S)$ be two topological dynamical systems. Suppose that  $(Y, S)$ is a factor of $(X, T)$,  in the sense that there exists a continuous surjective map $\pi: X\to Y$ such that $\pi\circ T=S\circ \pi$. The map $\pi$ is called a {\it factor map} from $X$ to $Y$.  Let $f$ be a real-valued continuous function on $X$, and let $a_1>0$, $a_2\geq 0$. The main purpose of this paper is to consider the following.

\begin{que}
\label{que-1}
How can one define a meaningful term $P^{(a_1, a_2)}(T, f)$ such that the following variational principle holds?
  \begin{equation}
 \label{e-var}
 P^{(a_1, a_2)}(T, f)=\sup\left\{a_1h_\mu(T)+a_2h_{\mu\circ\pi^{-1}}(S)+\int f \;d\mu\right\},
 \end{equation}
 where the supremum is taken over the set of all $T$-invariant Borel probability measures $\mu$ on $X$, and $h_\mu(T), h_{\mu\circ\pi^{-1}}(S)$ stand for the measure-theoretic entropies of $\mu$ and $\mu\circ \pi^{-1}$ with respect to $T$ and $S$, respectively (cf.  \cite{Wal82}).
\end{que}

  According to the variational principle of  Ruelle and Walters, the left-hand side of \eqref{e-var} equals $a_1P(T,\frac{1}{a_1}f)$ in the particular case when  $a_2=0$, where $P(T,\cdot)$ stands for the classic topological pressure of continuous functions (cf.  \cite{Wal82}). Our interest is on the general case that $a_2\neq 0$. This project is motivated from the study of dimension of invariant sets and measures on the tori  under diagonal affine expanding maps.

 Let $T$ be the endmorphism on the $2$-dimensional torus ${\Bbb T}^2=\R^2/\Z^2$
 represented by an integral diagonal matrix $A={\rm diag}(m_1,m_2)$, where $2\leq m_1<   m_2$. That is, $Tu=A u \;(\mbox{mod}\; 1)$ for $u\in  {\Bbb T}^2$. In their seminal works, Bedford \cite{Bed84} and McMullen \cite{McM84} independently determined the Hausdorff dimension of the so-called {\it self-affine Sierpinski gaskets}, which are a particular class of $T$-invariant subsets of $\T^2$ defined as follows:
 $$
 K(T, \D):=\left\{\sum_{n=1}^\infty A^{-n}u_n:\; u_n\in \D \mbox{ for all }n\geq 1\right\},
 $$
where   $\D$ runs over the non-empty subsets of $$\left\{\left(\begin{array}{l} i\\
 j\end{array} \right):\; i=0,1,\ldots, m_1,\; j=0,1,\ldots, m_2-1\right\}.$$   Moreover, McMullen \cite{McM84} exhibited explicitly that  for each $\D$, there exists an ergodic $T$-invariant measure $\mu$ supported on $ K(T, \D)$ with $\dim_H\mu=\dim_HK(T, \D)$, where $\dim_H$ denotes the Hausdorff dimension of a set or measure (cf. \cite{Fal03}). Later Kenyon and Peres \cite{KePe96} extended this result to any  compact $T$-invariant set $K\subseteq {\Bbb T}^2$, that is, there is an ergodic $T$-invariant measure $\mu$ supported on $K$ so that $\dim_H\mu=\dim_HK$. Furthermore Kenyon and Peres \cite{KePe96} established the following variational principle for the Hausdorff dimension of $K$:
\begin{equation}
\label{e-var-H}
\dim_HK=\sup\left\{ \frac{1}{\log m_2}h_\eta(T)+\left(\frac{1}{\log m_{1}}-\frac{1}{\log m_{2}}\right)h_{\eta\circ \pi^{-1}}(S) \right\},
\end{equation}
where the supremum is taken over the collection of $T$-invariant Borel probability measures $\eta$ supported on $K$,  $\pi: \T^2\to \T^1$ denotes the projection $(x,y)\mapsto x$, and $S: \T^1\to \T^1$ denotes the map $x\mapsto m_1 x(\mbox{mod}\; 1)$. It is easy to check that $(\T^1, S)$ is a factor of $(\T^2, T)$ with the factor map $\pi$. We emphasize that for any ergodic $T$-invariant measure $\eta$ on $\T^2$, the sum in the bracket of  \eqref{e-var-H} just equals $\dim_H\eta$ (cf. \cite[Lemma 3.1]{KePe96}); i.e.
\begin{equation}
\label{e-LY}
\dim_H\eta= \frac{1}{\log m_2}h_\eta(T)+\left(\frac{1}{\log m_{1}}-\frac{1}{\log m_{2}}\right)h_{\eta\circ \pi^{-1}}(S).
\end{equation}
 This is a version of  Ledrappier-Young  dimension formula for ergodic measures on $\T^2$.  We remark that an extension of the variational relation \eqref{e-var-H} to higher dimensional tori was also established by
Kenyon and Peres \cite{KePe96}.

Let us turn back to Question \ref{que-1}. According to \eqref{e-var-H}, if $\pi$ is the factor map $(x,y)\mapsto x$ between the toral dynamics
$(K, T)$ and $(\pi(K), S)$ as in the above paragraph, and if $f\equiv 0$ on $K$, and $a_1=\frac{1}{\log m_2}$, $a_2=\frac{1}{\log m_{1}}-\frac{1}{\log m_{2}}$,  then we can just define $P^{(a_1,a_2)}(f)$ to be the Hausdorff dimension of $K$. The problem arises how can we extend this to general factor maps between topological dynamical systems, as well as to general continuous functions $f$ and vectors $(a_1,a_2)$.

In \cite{BaFe12, Fen11}, Barral and the first author defined $P^{(a_1,a_2)}(f)$ (and called it {\it weighted topological pressure}) via relative thermodynamic formalism and subadditive thermodynamic formalism, in the particular case when the underlying dynamical systems $(X, T)$ and $(Y, S)$ are subshifts over finite alphabets. They also studied the dynamical properties of  weighted equilibrium measures (i.e. the invariant measures $\mu$ which attain the supremum in \eqref{e-var}) and gave the applications to the multifractal analysis on Sirpinski gaskets/sponges \cite{BaFe12}, and to the  uniqueness of invariant measures  of  full dimension supported on  affine-invariant subsets of tori \cite{Fen11}. Independently, in this subshift case Yayama \cite{Yay11} defined $P^{(a_1,a_2)}(f)$ for the particular case $f\equiv 0$, along the similar way.

However, the approach of \cite{BaFe12, Fen11} in defining $P^{(a_1,a_2)}(f)$ relies on certain special property of subshifts and does not extend to general topological dynamical systems (see Section~\ref{s-7.1} for details). Moreover, the variational principle
established therein does not give a new proof of Kenyon and Peres' variational relation \eqref{e-var-H} for the Hausdorff dimension.

In the paper, we define $P^{(a_1,a_2)}(f)$ in a new way, which is  inspired from the dimension theory of affine invariant subsets of tori, and from the ``dimension'' approaches of Bowen \cite{Bow73} and Pesin-Pitskel'
\cite{PePi84} in defining the topological entropy and topological pressure for arbitrary subsets.

 We will present our definition under a more general setting. Let $k\geq 2$. Assume that  $(X_i, d_i)$, $i=1,\ldots, k$, are compact metric spaces, and  $(X_i, T_i)$ are topological dynamical systems. Moreover, assume that for each $1\leq i\leq k-1$,   $(X_{i+1}, T_{i+1})$ is a factor of $(X_i, T_i)$ with a factor map $\pi_i: X_i\to X_{i+1}$; in other words, $\pi_1,\ldots, \pi_{k-1}$ are continuous maps so that the following diagrams commute.

$$
\begin{CD}
X_1 @ > \pi_1 >> X_2 @ >\pi_2>> \cdots @>\pi_{k-1}>> X_k\\
@V T_1VV  @ VV T_2 V @. @VV T_k V\\
X_1 @ > \pi_1 >> X_2 @ >\pi_2>> \cdots @>\pi_{k-1}>> X_k
\end{CD}
$$

For convenience,  we use $\pi_0$ to denote the identity map on $X_1$.
  Define
 $\tau_i:\;X_1\to X_{i+1}$ by $\tau_i=\pi_i\circ\pi_{i-1}\circ
\cdots \circ \pi_0$ for $i=0,1,\ldots,k-1$.

Let $\M(X_i,T_i)$ denote the set of all $T_i$-invariant Borel probability measures on $X_i$, endowed with the weak-star topology.
Fix $\ba=(a_1,a_2,\ldots,a_k)\in \R^k$ with $a_1>0$ and $a_i\geq 0$ for $i\geq 2$. For $\mu\in \M(X_1,T_1)$, we call
$$
h^{\ba}_\mu(T_1):=\sum_{i=1}^ka_ih_{\mu\circ \tau_{i-1}^{-1}}(T_i)
$$
the {\it $\ba$-weighted measure-theoretic entropy of $\mu$ with respect to $T_1$}, or simply, the {\it $\ba$-weighted entropy of $\mu$}, where
$h_{\mu\circ \tau_{i-1}^{-1}}(T_i)$ denotes the measure-theoretic entropy of $\mu\circ \tau_{i-1}^{-1}$ with respect to $T_i$.

\begin{de}[$\ba$-weighted Bowen ball]
\label{de-1.1}

  For $x\in X_1$, $n\in \N$, $\epsilon>0$,  denote
\begin{equation*}
\begin{split}
& B_n^\ba(x,\epsilon):=\left \{y\in X_1:\; d_i(T_i^j\tau_{i-1} x, T_i^j \tau_{i-1} y)< \epsilon \mbox{ for } 0\leq j\leq \lceil(a_1+\ldots +a_i)n\rceil-1, \right. \\
 &\left. \mbox{}\qquad\quad \qquad \qquad  i=1,\ldots, k\right\},
\end{split}
\end{equation*}
 where $\lceil u\rceil$ denotes the least integer $\geq u$. We call $B^\ba_n(x, \epsilon)$ the $n$-th $\ba$-weighted  Bowen ball of radius $\epsilon$ centered at  $x$.
\end{de}

Following the approaches of Bowen \cite{Bow73} and Pesin-Pitskel'
\cite{PePi84} in defining topological entropies and topological pressures of non-compact subsets \cite{Bow73}, and in which replacing Bowen balls by $\ba$-weighted Bowen balls, we can define the notions of  $\ba$-weighted topological entropy and $\ba$-weighted topological  pressure, respectively.  To be concise, in this section we only give the definition of $\ba$-weighted topological entropy. The definition of  $\ba$-weighted  topological pressure will be given in Section~\ref{s-3.1}.

Let $Z \subset X_1$ and $\epsilon> 0$. We say that an at
most countable  collection of $\ba$-weighted Bowen  balls $\Gamma
=\{B^\ba_{n_j}(x_j,\epsilon)\}_j$ {\it covers} $Z$ if $Z \subset \bigcup_j
B^\ba_{n_j}(x_j,\epsilon)$. For
$\Gamma=\{B^\ba_{n_j}(x_j,\epsilon)\}_j$, put $n(\Gamma) =\min_j
n_j$. Let $s \geq 0$ and define $$\Lambda^{\ba, s}_{N, \epsilon}(Z) =\inf\sum_j
\exp(-sn_j),$$
 where the infinum is taken over all collections $\Gamma=\{B^\ba_{n_j}(x_j,\epsilon)\}$ covering $Z$,
 such that $n(\Gamma) \geq  N$. The quantity $\Lambda^{\ba, s}_{N, \epsilon}(Z)$ does not decrease with $N$, hence the following
limit exists: $$\Lambda^{\ba, s}_\epsilon(Z) = \lim_{N\to \infty} \Lambda^{\ba, s}_{N, \epsilon}(Z).$$ There exists a critical value
of the parameter s, which we will denote by $\htop^\ba(T_1,Z,\epsilon)$,
where $\Lambda^{\ba, s}_{\epsilon}(Z)$ jumps from $\infty$ to $0$, i.e.
\[
\Lambda^{\ba, s}_{\epsilon}(Z) = \left\{
\begin{array}{ll}
0, & s > \htop^\ba(T_1,Z,\epsilon),\\
\infty,& s < \htop^\ba (T_1,Z,\epsilon).
\end{array}
\right.
\]
It is clear to see that $\htop^\ba(T_1,Z,\epsilon)$ does not decrease with $\epsilon$, and  hence the following limit exists, $$\htop^\ba(T_1, Z) = \lim_{\epsilon\to 0}\htop^\ba(T_1,Z,\epsilon).$$

 \begin{de}
 \label{de-1.2}
 We  call
$\htop^\ba(T_1, Z)$ the {\it $\ba$-weighted topological entropy of $T_1$ restricted to
$Z$} or, simply, the {\it $\ba$-weighted topological entropy of $Z$}, when there
is no confusion about $T_1$. In particular we write $\htop^\ba(T_1)$ for $\htop^\ba(T_1,X_1)$.
\end{de}

Similarly we will define the   $\ba$-weighted topological pressure $P^\ba(T_1, f)$ of continuous functions  $f$ on $X_1$ (see Section~\ref{s-3.1}). In the particular case when $f\equiv 0$, we have $P^\ba(T_1, 0)=\htop^\ba(T_1)$. The main result of this paper is the following variational principle for weighted topological pressure.

\begin{thm}\label{thm-1.1} Let $f\in C(X_1)$. Then
\begin{equation}
\label{e-main}
P^\ba(T_1, f)=\sup\left\{\int f  d\mu+h_\mu^\ba(T_1):\; \mu\in \M(X_1, T_1)\right\}.
\end{equation}
\end{thm}

In Section 6, we will extend the above theorem to the case that $f$ is a sub-additive potential.   As a corollary,  taking $f\equiv 0$ in Theorem \ref{thm-1.1}, we obtain the following  variational principle for  weighted topological entropy.

\begin{cor}
\label{cor-entropy}
$
\htop^\ba(T_1)=\sup\{h_\mu^\ba(T_1):\; \mu\in \M(X_1, T_1)\}.
$
\end{cor}

 Theorem \ref{thm-1.1} and Corollary \ref{cor-entropy} provide as weighted versions of Ruelle-Walters' variational principle for topological pressure, and Goodwyn-Dinaburg-Goodman's variational principle for topological entropy (cf. \cite{Wal82}).  They are also topological extensions of Kenyon-Peres' variational principle  for Hausdorff dimension of toral affine invariant sets. Indeed, consider the  aforementioned factor map $\pi$  between the toral dynamics $(K, T)$ and $(\pi(K), S)$ and let  $a_1=\frac{1}{\log m_2}$, $a_2=\frac{1}{\log m_{1}}-\frac{1}{\log m_{2}}$. It is easy to see from our definition that $\htop^{(a_1, a_2)}(T, K)$ simply coincides with $\dim_HK$, and hence Corollary \ref{cor-entropy} recovers \eqref{e-var-H} and its higher dimensional versions given in \cite{KePe96}. Moreover, by Corollary \ref{cor-entropy}, we can generalize \eqref{e-var-H} to a class of skew-product expanding maps on  the $k$-torus (see Section~\ref{s-7.2} for details).

The proof of Theorem  \ref{thm-1.1} is quite sophisticated. Besides adopting some ideas from \cite{Wal75, Mis76} and \cite{KePe96}, we also introduce substantially new ideas in the proof.  For the convenience of the readers, in the following we illustrate  a rough outline of  our proof.

To see the lower bound in \eqref{e-main}, we first prove that for each  ergodic measure $\mu\in \M(X_1, T_1)$,
\begin{equation}
\label{e-BK}
\lim_{\epsilon \rightarrow 0} \liminf_{n\rightarrow +\infty} \frac{-\log \mu(B_n^{\bf a}(x,\epsilon))}{n}=
\lim_{\epsilon \rightarrow 0} \limsup_{n\rightarrow +\infty}
\frac{-\log \mu(B_n^{\bf a}(x,\epsilon))}{n}=h_\mu^{\bf a}(T_1)
\end{equation}
for $\mu$-a.e. $x\in X_1$. The above formula is not only a weighted version of Brin-Katok's Theorem \cite{BrKa83}  on local entropy, but also a topological extension of the Ledrappier-Young dimension formula \eqref{e-LY}.  The justification of \eqref{e-BK} is mainly adapted from Kenyon-Peres' proof of \eqref{e-LY} in \cite{KePe96} and  Brin-Katok's argument in \cite{BrKa83}. Based on \eqref{e-BK}, the lower bound in \eqref{e-main} follows from a simple covering argument.

The proof of the upper bound in \eqref{e-main} is  more complicated. First we apply the techniques in geometric measure theory to prove the following ``dynamical'' Frostman lemma: for any $0<s<P^{\ba}(T_1,f)$ and small enough $\epsilon>0$, there exists a Borel probability measure $\nu$ on $X_1$ and $N\in \N$ such that
\begin{equation}
\label{e-dyn}
\nu(B_n^{\ba}(x,\epsilon))\leq
\exp\left(-sn + \frac{1}{a_1}S_{\lceil a_1n\rceil}f(x)\right), \quad \forall x\in X_1, \; n\geq N,
\end{equation}
where $S_nf(x):=\sum_{i=0}^{n-1}f(T_1^i x)$. This is a key part in our proof. Notice that there  exists a small $\tau\in (0,\epsilon)$ such that for any Borel partition $\alpha_i$ of $X_i$ with $\mbox{diam}(\alpha_i)<\tau$, $i=1,\ldots, k$, we have
$$
\bigvee_{i=1}^k\bigvee_{j=t_{i-1}(n)}^{t_i(n)-1}T^{-j}_1\pi^{-1}_{i-1} \alpha_i(x)\subseteq B^\ba_n(x,\epsilon), \quad \forall x\in X_1, \; n\geq N,
$$
where $t_0(n)=0$, $t_i(n)=\lceil (a_1+\ldots +a_{i})n\rceil$, and $\vee$ stands for the join of partitions.
Hence \eqref{e-dyn} implies that
\begin{equation}
\label{e-6.2''}
\sum_{i=1}^k H_{\nu}\Big(\bigvee_{j=t_{i-1}(n)}^{t_i(n)-1}T^{-j}_1\pi^{-1}_{i-1} \alpha_i\Big)\geq sn  -\int \frac{1}{a_1}S_{\lceil a_1n\rceil}f(x) d\nu(x).
\end{equation}
Then,   as another key part, we  use \eqref{e-6.2''} and entropy theory to show  the  existence of a $T_1$-invariant measure $\mu$ on $X_1$ such that  $h^{\ba}_\mu(T_1)>s-\int f d\mu$, from which the upper bound follows. In the proof of this part, a combinatoric lemma (see Lemma \ref{lem-KP}) established by Kenyon-Peres \cite{KePe96} plays an important role; besides this, we also use a delicate compactness argument based on the upper semi-continuity of certain entropy functions, and adopt some ideas from \cite{Wal75,Mis76} as well. Reducing back to the aforementioned toral dynamics, our approach provides a new proof for the upper bound in Kenyon-Peres' variational principle \eqref{e-var-H}.

The paper is organized as follows. In Section~\ref{s-2}, we prove the upper semi-continuity of certain entropy functions.  In Section~\ref{s-3}, we define  weighted topological pressure for continuous functions and more generally for sub-additive potentials; we also establish a dynamical Frostman lemma for  the weighted topological pressure.    In Sections \ref{s-4}-\ref{s-5}, we prove respectively the lower and upper bounds of  Theorem \ref{thm-1.1}. In Section~\ref{s-6}, we extend Theorem \ref{thm-1.1} to the sub-additive case. In Section~7, we give some remarks, examples and questions.  In Appendix~\ref{s-a}, we prove  the formula \eqref{e-BK}.

\section{Upper semi-continuity of certain entropy functions}
\label{s-2}
In this section, we prove the upper semi-continuity of certain entropy functions (see Lemma \ref{usc}), which is needed in our proof of the upper bound part of Theorem \ref{thm-1.1}. We begin with the following.
\begin{de}
Let $Z$ be a compact metric space. A  function $f:Z\rightarrow
[-\infty,+\infty)$ is called upper semi-continuous if one of the
following equivalent conditions holds:
\begin{enumerate}
\item[{(C1)}] $\limsup\limits_{z_N\rightarrow z} f (z_N) \leq f (z)$ for each $z \in Z$;

\item[{(C2)}] for each $r\in \mathbb{R}$ the set $\{z \in Z : f (z) \ge r\}$ is closed.
\end{enumerate}
\end{de}
By (C2), the infimum of any family of upper semi-continuous functions
is again an upper semi-continuous function; both the sum and supremum of
finitely many upper semi-continuous functions are upper
semi-continuous functions.

\begin{lem}\label{appro} Let  $Z$ be a compact metric space and $f:Z\rightarrow
[-\infty,+\infty)$ be an upper semi-continuous function. Then for any
$\mu\in \mathcal{M}(Z)$,
\begin{equation}
\label{usc-1}\inf \limits_{g\in C(Z),g\ge f}\int_Z
g(z) d \mu(z)=\int_Z f(z) d \mu(z).
\end{equation}
\end{lem}
\begin{proof} It is well known that  the equality \eqref{usc-1} holds when $f$ is a real-valued upper semi-continuous function (see e.g. \cite[Appendix (A7)]{DS} for a proof). In the following we assume that $f$ is an upper semi-continuous function taking values in $[-\infty,+\infty)$.

By the upper semi-continuity of $f$, we have
$\sup_{z\in Z}f(z)=\max_{z\in Z}f(z)<+\infty$. Thus $\int_Z f(z) d
\mu(z)$ is well defined and $\int_Z f(z) d \mu(z)\in
[-\infty,+\infty)$.

For $M\in \mathbb{N}$, let $f_M(z)=\max\{ f(z),-M\}$ for $z\in Z$.
 Then $f_M$ is an upper semi-continuous real-valued function, and thus
 $$\inf \limits_{g\in C(Z),g\ge f_M}\int_Z
g(z) d \mu(z)=\int_Z f_M(z) d \mu(z).$$  Since
$$\sup \limits_{M\in
\mathbb{N}}\sup \limits_{z\in Z} f_M(z)\le \max\left\{\max \limits_{z\in
Z}f(z),0\right\}<+\infty$$ and $ f_M(z)\searrow f(z)$ as $M\rightarrow
+\infty$ for any $z\in Z$, one has
$$\lim_{M\rightarrow +\infty} \int_Z f_M(z) d \mu(z)= \int_Z
\lim_{M\rightarrow +\infty} f_M(z) d \mu(z) = \int_Z  f(z) d
\mu(z)$$ by Lebesgue's monotone convergence theorem. Moreover
\begin{align*}
\inf \limits_{g\in C(Z),g\ge f}\int_Z g(z) d \mu(z)&=\inf_{M\in
\mathbb{N}} \left\{\inf\limits_{g\in C(Z),g\ge f_M} \int_Z g(z) d
\mu(z)\right\}\\
&=\inf_{M\in \mathbb{N}} \int_Z f_M(z) d
\mu(z) \\
&=\lim_{M\rightarrow +\infty} \int_Z f_M(z) d \mu(z)\\
&= \int_Z  f(z) d
\mu(z).
\end{align*}
This completes the proof of the lemma.
\end{proof}

Let $(X,T)$ be a TDS with a compatible metric $d$. For $\epsilon>0$
and $M\in \mathbb{N}$, we define
\begin{equation}
\label{e-2014-6}
\mathcal{P}_X(\epsilon,M)=\{ \alpha: \alpha \text{ is a finite
Borel partition of }X \text{ with } \text{diam}(\alpha)<\epsilon,
\#(\alpha)\le M\},
\end{equation}
where $\text{diam}(\alpha):=\max_{A\in \alpha}\text{diam}(A)$, and
$\#(\alpha)$ stands for the cardinality of $\alpha$.
 Then we define
$$\mathcal{P}_X(\epsilon)=\{\alpha: \alpha \text{ is a finite
Borel partition of }X \text{ with }
\text{diam}(\alpha)<\epsilon\}.$$ It is clear that for any
$\epsilon>0$, there exists $N:=N(\epsilon)\in \mathbb{N}$ such that
$\mathcal{P}_X(\epsilon,M)\neq\emptyset$ for any $M\ge N$. The main result of this section is the following.

\begin{lem}\label{usc} Let $(X,T)$ be a TDS and $\epsilon>0$. Then
\begin{enumerate}
\item  If  $M\in \mathbb{N}$ with
$\mathcal{P}_X(\epsilon,M)\neq \emptyset$, then the map
\begin{equation}
\label{e-HEML}
\theta\in
\mathcal{M}(X)\mapsto H_\theta(\epsilon,M;\ell):=\inf_{\alpha\in
\mathcal{P}_X(\epsilon,M)}
\frac{1}{\ell}H_\theta\left(\bigvee_{i=0}^{\ell-1}T^{-i}\alpha\right)
\end{equation} is
upper semi-continuous from $\mathcal{M}(X)$ to $[0,\log M]$ for
each $\ell\in \mathbb{N}$.

\item The map $$\theta\in
\mathcal{M}(X)\mapsto H_\theta(\epsilon;\ell):=\inf_{\alpha\in
\mathcal{P}_X(\epsilon)}
\frac{1}{\ell}H_\theta\left(\bigvee_{i=0}^{\ell-1}T^{-i}\alpha\right)$$ is a
bounded upper semi-continuous non-negative function  for each
$\ell\in \mathbb{N}$.

\item The map $$\mu\in \mathcal{M}(X,T)\mapsto
h_\mu(T,\epsilon):=\inf_{\alpha\in
\mathcal{P}_X(\epsilon)}h_\mu(T,\alpha)$$ is a bounded upper
semi-continuous non-negative function.
\end{enumerate}
\end{lem}
\begin{proof} We first prove (1). Let $M\in \mathbb{N}$ with
$\mathcal{P}_X(\epsilon,M)\neq \emptyset$, and  $\ell\in
\mathbb{N}$. Clearly,  the map $H_{\bullet}(\epsilon,M;\ell)$ is defined from
$\mathcal{M}(X)$ to $[0,\log M]$.
Let $\theta_0\in \mathcal{M}(X)$. It is sufficient to show
that the map $H_{\bullet}(\epsilon,M;\ell)$ is upper semi-continuous at
$\theta_0$.

Let  $\delta>0$. Then there exists $\alpha\in
\mathcal{P}_X(\epsilon,M)$ such that
\begin{equation}
\label{e-diff}
\frac{1}{\ell}H_{\theta_0}\left(\bigvee_{i=0}^{\ell-1}T^{-i}\alpha\right)\le
H_{\theta_0}(\epsilon,M;\ell)+\delta.
\end{equation}
 Let $\alpha=\{
A_1,\ldots,A_u\}$. Then $u\le M$ and $\text{diam}(A_i)<\epsilon$
for $i=1,2,\ldots,u$. By Lemma 4.15 in \cite{Wal82}, there exists
$\delta_1= \delta_1(u,\delta)
> 0$ such that whenever $\gamma_1=\{E_1,\ldots, E_u\},\gamma_2=\{F_1,\ldots, F_u\}$  are two Borel partitions
of $X$ with $\sum_{j=1}^u \sum_{i=0}^{\ell-1} \theta_0\circ
T^{-i}(E_j\Delta F_j)<\delta_1$,  then
\begin{equation}
\label{entropy}
\begin{split}
\frac{1}{\ell}& \left|H_{\theta_0}\left(\bigvee_{i=0}^{\ell-1}T^{-i}\gamma_1\right)-H_{\theta_0}\left(\bigvee_{i=0}^{\ell-1}T^{-i}\gamma_2\right)\right|
\\
&\le \frac{1}{\ell}\sum_{i=0}^{\ell-1} \left|H_{\theta_0\circ
T^{-i}}(\gamma_1|\gamma_2)+H_{\theta_0\circ
T^{-i}}(\gamma_2|\gamma_1)\right|<\delta.
\end{split}
\end{equation}
 Write
$\eta=\sum_{i=0}^{\ell-1} \theta_0\circ T^{-i}$. Next, we are going to
construct a Borel partition $\beta=\{ B_1,\ldots,B_u\}$ of $X$ so that
$\text{diam}(\beta)<\epsilon$, $\sum_{j=1}^u \eta(A_j\Delta
B_j)<\delta_1$ and $\eta(\partial \beta)=0$.

In fact, note that $\eta(X)=\ell<\infty$, hence $\eta$ is regular on
$X$. Thus there exist open subsets $V_j$ of $X$
such that $A_j\subseteq V_j$, $\text{diam}(V_j)<\epsilon$ and
$\eta(V_j\setminus A_j)<\frac{\delta_1}{u^2}$ for $j=1,\ldots,u$.
Clearly, $\mathcal{V}:=\{ V_1,\ldots,V_u\}$ is an open cover.
Let $t>0$ be a Lebesgue number of $\mathcal{V}$. For any $x\in X$,
there exists $0<t_x\le \frac{t}{3}$ such that $\eta(\partial
B(x,t_x))=0$.
  Thus $\{
B(x,t_x):x\in X\}$ forms an open cover of $X$. Take its finite
subcover $\{ B(x_i,t_{x_i})\}_{i=1}^r$, that is, $\bigcup_{i=1}^r
B(x_i,t_{x_i})=X$. Obviously, each $B(x_i,t_{x_i})$ is a subset of
some $V_{j(i)}$, $j(i)\in \{1,\ldots,u\}$ since $t_{x_i}\le
\frac{t}{3}$.

Let $I_j=\{ i\in \{1,\ldots,r\}: B(x_i,t_{x_i})\subset V_j\}$ for
$j=1,\ldots,u$. Then $\bigcup_{j=1}^u I_j=\{1,\ldots,r\}$. Put
$B_1=\bigcup_{i\in I_1} B(x_i,t_{x_i})$ and $B_j=\left(\bigcup_{i\in I_j}
B(x_i,t_{x_i})\right)\setminus \bigcup_{m=1}^{j-1}B_m$ inductively for
$j=2,\ldots,u$. It is clear that $\beta=\{ B_1,\ldots,B_u\}$ is
a Borel partition of $X$ with $B_j\subseteq V_j$ and $\eta(\partial
B_j)=0$ for $j=1,\ldots,u$. Now for each $j\in \{1,\ldots,u\}$,
\begin{align*}
A_j\Delta B_j&= (B_j\setminus A_j)\cup (A_j\cap (X\setminus
B_j))\subseteq (V_j\setminus A_j)\cup \bigcup_{k\neq j}(A_j \cap
B_k)\\
&\subseteq (V_j\setminus A_j)\cup \bigcup_{k\neq j}(A_j \cap
V_k)\subseteq (V_j\setminus A_j)\cup \bigcup_{k\neq j}(A_j \cap
(V_k\setminus A_k))\\
&\subseteq \bigcup_{k=1}^u (V_k\setminus A_k).
\end{align*}
Thus $\sum_{j=1}^u \eta(A_j\Delta
B_j)\le u\sum_{k=1}^u \eta (V_k\setminus A_k)<\delta_1$.

Summing up, we have constructed a Borel partition $\beta=\{B_1,\ldots, B_u\}\in \mathcal{P}_X(\epsilon,M)$ so that
$\sum_{j=1}^u\eta(B_j\Delta A_j)<\delta_1$ and $\eta(\partial \beta)=0$.
Now on the one hand,  by \eqref{entropy} and \eqref{e-diff}, we have
$$\frac{1}{\ell}H_{\theta_0}\left(\bigvee_{i=0}^{\ell-1}T^{-i}\beta\right)\le
\frac{1}{\ell}H_{\theta_0}\left(\bigvee_{i=0}^{\ell-1}T^{-i}\alpha\right)+\delta\le
H_{\theta_0}(\epsilon,M;\ell)+2\delta.$$ On the other hand, since
$\eta(\partial \beta)=0$, one has $\theta_0(T^{-i}\partial \beta)=0$
for $i=0,1,\cdots,\ell-1$. As $ \partial T^{-i}A\subseteq
T^{-i}\partial A$ for any $A\subseteq X$, one has $\theta_0(\partial
T^{-i}\beta)=0$ for $i=0,1,\cdots,\ell-1$. Moreover note that
$\partial(A\cap B)\subseteq (\partial A)\cap (\partial B)$ for any
$A,B\subseteq X$, we have
$\theta_0(\partial(\bigvee_{i=0}^{\ell-1}T^{-i}\beta))=0$. Thus the
map $\theta\in \mathcal{M}(X)\mapsto
\frac{1}{\ell}H_\theta(\bigvee_{i=0}^{\ell-1}T^{-i}\beta)$ is
continuous at the point $\theta_0$. Therefore
\begin{align*}
\limsup_{\theta\rightarrow \theta_0} H_\theta(\epsilon,M;\ell)&\le
 \limsup_{\theta\rightarrow \theta_0}
\frac{1}{\ell}H_\theta \left(\bigvee_{i=0}^{\ell-1}T^{-i}\beta\right)\\
&=\frac{1}{\ell}H_{\theta_0}\left(\bigvee_{i=0}^{\ell-1}T^{-i}\beta\right)\\
&\le H_{\theta_0}(\epsilon,M;\ell)+2\delta.
\end{align*}
Finally letting $\delta\searrow 0$, we see that
the map $H_{\bullet}(\epsilon,M;\ell)$ is upper semi-continuous at
$\theta_0$. This completes the proof of (1).

\medskip
Now we turn to the proof of (2). Let $\ell\in \mathbb{N}$. Since
$\mathcal{P}_X(\epsilon)=\bigcup_{M\in
\mathbb{N},\mathcal{P}_X(\epsilon,M)\neq
\emptyset}\mathcal{P}_X(\epsilon,M)$, we have
$$H_\theta(\epsilon;\ell)=\inf_{M\in
\mathbb{N},\mathcal{P}_X(\epsilon,M)\neq
\emptyset}H_\theta(\epsilon,M;\ell)$$ for $\theta\in
\mathcal{M}(X)$. Moreover, by (1) and the fact that the infimum of
any family of upper semi-continuous functions is again an upper
semi-continuous one, we know that the map
$$\theta\in \mathcal{M}(X)\mapsto
H_\theta(\epsilon;\ell):=\inf_{\alpha\in \mathcal{P}_X(\epsilon)}
\frac{1}{\ell}H_\theta\left(\bigvee_{i=0}^{\ell-1}T^{-i}\alpha\right)$$ is a
bounded upper semi-continuous non-negative function. This proves (2).

\medskip
In the end we prove (3). Note that
\begin{align*}
h_\mu(T,\epsilon)&=\inf \limits_{\alpha\in
\mathcal{P}_X(\epsilon)}h_\mu(T,\alpha)=\inf \limits_{\alpha\in
\mathcal{P}_X(\epsilon)} \inf_{\ell\ge
1}\frac{1}{\ell}H_\mu\left(\bigvee_{i=0}^{\ell-1}T^{-i}\alpha\right)\\
&=\inf_{\ell\ge 1}\inf \limits_{\alpha\in \mathcal{P}_X(\epsilon)}
\frac{1}{\ell}H_\mu\left(\bigvee_{i=0}^{\ell-1}T^{-i}\alpha\right)=\inf_{\ell\ge
1}H_\mu(\epsilon;\ell)
\end{align*}
 for $\mu\in \mathcal{M}(X,T)$.  Using (2) and
 the fact that the infimum of any family of upper semi-continuous
functions is again an upper semi-continuous one, we know that the map
$$\mu\in \mathcal{M}(X,T)\mapsto
h_\mu(T,\epsilon)$$ is a bounded upper semi-continuous non-negative
function. This completes the proof of the lemma.
\end{proof}

\section{Weighted topological pressures and a dynamical Frostman lemma}\label{s-3}

In this section we introduce the definition of weighted topological pressure for (asymptotically) sub-additive potentials for general topological dynamical systems. Moreover, using some ideas from geometric measure theory, we establish a dynamical Frostman lemma (see Lemma \ref{lem-Frost}) for weighted topological pressure, which plays a key role in our proof of Theorem \ref{thm-1.1}.

\subsection{Weighted topological pressures for sub-additive potentials}
\label{s-3.1}
Assume that $(X, T)$ is a TDS. We say that a sequence   $\Phi=\{\log \phi_n\}_{n=1}^\infty$ of functions on $X$ is  a {\it sub-additive potential} if
each $\phi_n$ is an upper semi-continuous nonnegative-valued function on $X$ such that
\begin{equation}
\label{e-1.0}
 0\leq \phi_{n+m}(x)\leq \phi_n(x)\phi_m(T^nx),\qquad \forall\;
x\in X, \; m,n\in \N.
\end{equation}
In particular,  $\Phi$ is called {\it additive} if each $\phi_n$ is a continuous positive-valued function so that
$\phi_{n+m}(x)=\phi_n(x)\phi_m(T^nx)$ for all $x\in X$ and $m,n\in \N$; in this case, there is a
continuous real function $g$
on $X$ such that $\phi_n(x)=\exp(\sum_{i=0}^{n-1}g(T^ix))$ for each $n$.

Let $k\geq 2$. Assume that  $(X_i, d_i)$, $i=1,\ldots, k$, are compact metric spaces, and  $(X_i, T_i)$ are TDS's. Moreover, assume that for each $1\leq i\leq k-1$,   $(X_{i+1}, T_{i+1})$ is a factor of $(X_i, T_i)$ with a factor map $\pi_i: X_i\to X_{i+1}$.

Let $\ba=(a_1,\ldots, a_k)\in \R^k$ with $a_1>0$ and $a_i\geq 0$ for $2\leq i\leq k$.
For any $n\in \N$ and $\epsilon>0$, define
\begin{equation}
\label{e-0.1}
{\mathcal T}^\ba_{n,\epsilon}:=\{A\subset X_1:\; A \mbox { is Borel subset of } B^\ba_n(x,\epsilon) \mbox{ for some }x\in X_1\},
\end{equation}
where $B^\ba_n(x,\epsilon)$ is defined as in Definition \ref{de-1.1}.

 Let $\Phi=\{\log \phi_n\}_{n=1}^\infty$ be a sub-additive potential on $X_1$.  Let $Z\subseteq X_1$, $s \geq 0$ and $N\in \N$, define $$\Lambda^{\ba, s}_{\Phi, N, \epsilon}(Z) =\inf\sum_j
\exp\left(-sn_j+\frac{1}{a_1}\sup_{x\in A_j}\phi_{\lceil a_1n_j\rceil}(x)\right ),$$
 where  the infimum is taken over all countable collections $\Gamma=\{(n_j, A_j)\}_j$ with $n_j\geq N$, $A_j\in {\mathcal T}^\ba_{n_j,\epsilon}$ and $\bigcup_jA_j\supseteq Z$.
  The quantity $\Lambda^{\ba, s}_{\Phi, N, \epsilon}(Z) $ does not decrease with $N$, hence the following
limit exists: $$\Lambda^{\ba, s}_{\Phi, \epsilon}(Z)  = \lim_{N\to \infty} \Lambda^{\ba, s}_{\Phi, N, \epsilon}(Z) .$$ There exists a critical value
of the parameter $s$, which we will denote by $P^\ba(T_1,\Phi, Z, \epsilon)$,
where $\Lambda^{\ba, s}_{\Phi, \epsilon}(Z)$ jumps from $\infty$ to $0$, i.e.
\[
\Lambda^{\ba, s}_{\Phi,  \epsilon}(Z)  = \left\{
\begin{array}{ll}
0, & s > P^\ba (T_1,\Phi,Z,\epsilon),\\
\infty,& s < P^\ba (T_1,\Phi, Z, \epsilon).
\end{array}
\right.
\]
Clearly  $P^\ba(T_1,\Phi,Z, \epsilon)$ does not decrease with $\epsilon$, and  hence the following limit exists, $$P^\ba (T_1,\Phi, Z) = \lim_{\epsilon\to 0}P^\ba (T_1,\Phi, Z, \epsilon).$$

\begin{de} We  call
$P^\ba (T_1,\Phi):=P^\ba (T_1,\Phi, X_1)$ the {\it $\ba$-weighted topological pressure of $\Phi$ with respect to
$T_1$} or, simply, the {\it $\ba$-weighted topological pressure of $\Phi$}, when there
is no confusion about $T_1$.
\end{de}

\begin{de}
Let $f\in C(X_1)$. Define $\Phi=\{\log \phi_n\}_{n=1}^\infty$ by $\phi_n(x)=\exp(\sum_{j=0}^{n-1}f(T_1^jx))$. In this case, $\Phi$ is additive. We just define
 $P^\ba (T_1,f):=P^\ba (T_1,\Phi)$.
\end{de}

Taking $f\equiv 0$, one can see that  $P^\ba (T_1,0)=\htop^\ba(T_1)$.  Let  $\Phi=\{\log \phi_n\}_{n=1}^\infty$ be a sub-additive potential on $X_1$. For any $\mu\in \M(X_1, T_1)$, define
\begin{equation} \label{e-1.2} \Phi_*(\mu):=\lim_{n\to
\infty}\int \frac{\log \phi_n(x)}{n}\; d\mu(x).
\end{equation}
 This limit always exists and takes values in $\R\cup \{-\infty\}$ (cf. \cite[Theorem 10.1]{Wal75}).

In our proof of Theorem \ref{thm-1.1}, we need the following dynamical Frostman lemma.
\begin{lem}
\label{lem-Frost}
 Let $\Phi=\{\log \phi_n\}_{n=1}^\infty$ be a sub-additive potential on $X_1$.  Suppose that $P^{\ba}(T_1,\Phi)>0$. Then  for any $0<s<P^{\ba}(T_1,\Phi)$, there exist a Borel probability measure $\nu$ on $X_1$ and $\epsilon>0$, $N\in \N$ such that for any $x\in X_1$ and $n\geq N$ we have
\begin{equation}
\label{e-dyn**}
\nu(B_n^{\ba}(x,\epsilon))\leq
 \exp (-sn)\sup_{y\in B_n^{\ba}(x,\epsilon)} (\phi_{\lceil a_1n\rceil}(y))^{1/a_1}.
\end{equation}
\end{lem}

A non-weighted version of the above lemma was first proved by the authors  in the particular case when $\phi_n\equiv 1$ (see \cite[Lemma 3.4]{FeHu12}), using some ideas and techniques in geometric measure theory.
In the remainder  of this section, we will give the detailed proof of Lemma \ref{lem-Frost}, by adapting and elaborating the approach in \cite{FeHu12}. A key ingredient of our proof is the notion of  average weighted topological pressure, which is an analogue of weight Hausdorff measure in geometric measure theory.   The definition of this notion and some of its properties will be given in next subsection. In Subsection~\ref{S-3}, we prove Lemma \ref{lem-Frost}.

\subsection{Average weighted topological pressures}
\label{s-average}

 Let $\Phi=\{\log \phi_n\}_{n=1}^\infty$ be a sub-additive potential on $X_1$. For any function $f: X_1\to [0, \infty)$, for $s \geq 0$ and $N\in \N$, define
 \begin{equation}\label{e-19}
  \W_{\Phi, N, \epsilon}^{\ba, s} (f) =\inf\sum_j
c_j\exp\left(-sn_j+\frac{1}{a_1}\sup_{x\in A_j}\log \phi_{\lceil a_1n_j\rceil}(x)\right ),
\end{equation}
 where  the infimum is taken over all countable collections $\Gamma=\{(n_j, A_j, c_j)\}_j$ with $n_j\geq N$, $A_j\in {\mathcal T}^\ba_{n_j,\epsilon}$, $0<c_j<\infty$, and $$\sum_{j} c_j\chi_{A_j}\geq f,$$
 where $\chi_A$ denotes the characteristic function of $A$, i.e., $\chi_A(x)=1$ if $x\in A$ and $0$ if $x\in X_1\backslash A$.

 For $Z\subseteq X_1$, we set $\W^{\ba,s}_{\Phi, N,\epsilon}(Z)=\W^{\ba,s}_{\Phi, N,\epsilon}(\chi_Z)$.
  The quantity $\W^{\ba,s}_{\Phi, N,\epsilon}(Z)$ does not decrease with $N$, hence the following
limit exists: $$\W^{\ba,s}_{\Phi, \epsilon}(Z) = \lim_{N\to \infty} \W^{\ba,s}_{\Phi, N,\epsilon}(Z).$$ There exists a critical value
of the parameter $s$, which we will denote by $P^\ba_{W}(T_1,\Phi, Z, \epsilon)$,
where $\W^{\ba,s}_{\Phi, \epsilon}(Z)$ jumps from $\infty$ to $0$, i.e.
\[
\W^{\ba,s}_{\Phi, \epsilon}(Z) = \left\{
\begin{array}{ll}
0, & s > P^\ba_{W} (T_1,\Phi, Z, \epsilon),\\
\infty,& s < P^\ba_{W} (T_1,\Phi,Z, \epsilon).
\end{array}
\right.
\]
Clearly  $P^\ba_{W} (T_1,\Phi, Z, \epsilon)$ does not decrease with $\epsilon$, and  hence the following limit exists, $$P^\ba_{W} (T_1,\Phi, Z) = \lim_{\epsilon\to 0}P^\ba_{W} (T_1,\Phi, Z, \epsilon).$$

\begin{de} We  call
$P^\ba_{W} (T_1,\Phi):= P^\ba_{W} (T_1,\Phi, X_1) $ the {\it average  $\ba$-weighted topological pressure of $\Phi$ with respect to
$T_1$} or, simply, the {\it  average $\ba$-weighted topological pressure of $\Phi$}, when there
is no confusion about $T_1$.
\end{de}

The main result of this subsection is the following.

\begin{pro}
\label{pro-2.1}
Let  $Z\subseteq X_1$.  Then for any $s\geq 0$ and  $\epsilon,\delta>0$, we have
\begin{equation*}
\label{e-ine}
\Lambda^{\ba, s+\delta}_{\Phi, N,6\epsilon}(Z)\leq \W^{\ba,s}_{\Phi, N,\epsilon}(Z)\leq \Lambda^{\ba, s}_{\Phi, N,\epsilon}(Z),
\end{equation*}
when $N$ is large enough. As a consequence, $P^\ba_{W} (T_1,\Phi)=P^\ba (T_1,\Phi)$.
\end{pro}

Before giving the proof of Proposition \ref{pro-2.1}, we first state some lemmas.
\begin{lem}
\label{lem-1.1}
 For any $s\geq 0$, $N\in \N$ and $\epsilon>0$, both $\Lambda^{\ba,s}_{\Phi, N,\epsilon}$ and $\W^{\ba, s}_{\Phi, N,\epsilon}$ are outer measures on $X$.
%\item[(ii)] For any $s\geq 0$, both $\M^s$ and $\W^s$ are metric outer measures on $X$.
\end{lem}
\begin{proof}
It follows directly from the definitions  $\Lambda^{\ba,s}_{\Phi, N,\epsilon}$ and $\W^{\ba, s}_{\Phi, N,\epsilon}$.
\end{proof}
%We remark that $\M^s$ and $\W^s$ depend not only $s$ but also the TDS $(X,T)$. However, $\M^s$ and $\W^s$ are purely topological and  independent of the special choice of %the metric $d$.

The following combinatoric lemma plays an important role in the proof of Proposition \ref{pro-2.1}.

\begin{lem}
 \label{lem-2.0}
 Let $(X, d)$ be a compact metric space and  $\epsilon>0$. Let  $(E_i)_{i\in \I}$ be a finite or countable family of subsets of $X$ with diameter less than $\epsilon$, and $(c_i)_{i\in \I}$ a family  of positive numbers. Let $t>0$. Assume that $F\subseteq X$ such that
$$
F\subseteq \left\{x\in X:\; \sum_i c_i\chi_{E_i} >t\right \}.
$$
Then $F$ can be covered by no more than
$\frac{1}{t}\sum_ic_i$ balls with centers in $\bigcup_{i\in \I} E_i$ and radius $6\epsilon$.
\end{lem}

To prove Lemma \ref{lem-2.0}, we need the following well known  covering lemma.
\begin{lem}
[cf. Theorem 2.1 in \cite{Mat95}]
\label{lem-2.1}

Let $(X, d)$ be a compact metric space and ${\mathcal
B}=\{B(x_i,r_i)\}_{i\in \mathcal I}$ be a family of  open
balls in $X$. Then there exists a finite or countable subfamily
${\mathcal B'}=\{B(x_i,r_i)\}_{i\in {\mathcal I}'}$ of pairwise
disjoint balls in ${\mathcal B}$ such that
$$\bigcup_{B\in {\mathcal B}} B\subseteq \bigcup_{i\in {\mathcal I}'}B(x_i,5r_i).$$
\end{lem}

\begin{proof}[Proof of Lemma \ref{lem-2.0}] Without loss of generality, assume that $\I\subseteq \N$. For any $i\in \I$, pick $x_i\in E_i$ and write $B_i=B(x_i, \epsilon)$ and
$5B_i=B(x_i, 5\epsilon)$ for short. Clearly $E_i\subseteq B_i$. Define
$$
Z=\left\{x\in X:\; \sum_i c_i\chi_{B_i} >t\right \}.
$$
We have $F\subset Z$. To prove the lemma, it suffices to show that $Z$ can be covered by no more than
$\frac{1}{t}\sum_ic_i$ balls with centers in $\{x_i:i\in \I\}$ and radius $6\epsilon$. To avoid triviality, we assume that $\sum_ic_i<\infty$; otherwise there is nothing left to prove.

 For $k\in \N$, define  $$\I_{k}=\{i\in \I:\; i\leq k\} \quad\mbox{and}\quad
 Z_{k}=\Big\{x\in Z:\; \sum_{i\in \I_{k}}c_i\chi_{B_i}(x)>t\Big\}.
  $$
 We divide the remaining  proof  into two small steps.

  {\sl Step 1. For each  $k\in \N$, there exists a finite set $\J_{k}\subseteq \I_{k}$ such that
 the balls $B_i$ {\rm ($i\in \J_{k}$)} are pairwise disjoint, $Z_{k}\subseteq \bigcup_{i\in \J_{k}}5B_i$ and
 $$\#(\J_{k})\leq \frac{1}{t}\sum_{i\in \I_{k}}c_i.$$
 }
  To prove the above result, we adopt the argument from Federer \cite[2.10.24]{Fed69} in the study of weighted Hausdorff measures (see also Mattila \cite[Lemma 8.16]{Mat95}). Since  $\I_{k}$ is finite, by approximating the $c_i$'s from above,
 we may assume that each $c_i$ is a positive rational, and then multiplying $c_i$ and $t$ with a common denominator we may assume that each $c_i$ is
 a positive integer. Let $m$ be the least integer with $m\geq t$. Denote $\B=\{B_i,\; i\in \I_{k}\}$ and define $u: \B\to \N$ by $u(B_i)=c_i$.
 We define by induction integer-valued functions $v_0,v_1,\ldots, v_m$ on $\B$ and sub-families $\B_1,\ldots, \B_m$ of $\B$ starting with $v_0=u$.
 Using Lemma \ref{lem-2.1} we  find a pairwise disjoint subfamily $\B_1$ of $\B$ such that
 $\bigcup_{B\in \B}B\subseteq \bigcup_{B\in \B_1} 5B$, and hence $Z_{k}\subseteq \bigcup_{B\in \B_1} 5B$. Then by repeatedly using Lemma \ref{lem-2.1}, we can  define inductively for $j=1, \ldots, m$,  disjoint subfamilies $\B_j$ of $\B$ such that
 $$\B_j\subseteq \{B\in \B:\; v_{j-1}(B)\geq 1\},\quad Z_{k}\subseteq \bigcup_{B\in \B_j} 5B$$
 and the functions $v_j$ such that
 $$
 v_j(B)=\left\{\begin{array}{ll}
 v_{j-1}(B)-1 & \mbox { for } B\in \B_j,\\
  v_{j-1}(B) & \mbox { for } B\in \B\backslash \B_j.
\end{array}
\right.
$$
This is possible since for $j<m$,
$Z_{k}\subseteq \big\{x: \sum_{B\in \B:\; B\ni x}v_j(B)\geq m-j\big\}$,
whence every $x\in Z_{k}$ belongs to some ball $B\in \B$ with $v_j(B)\geq 1$. Thus
\begin{eqnarray*}
 \sum_{j=1}^m \#(\B_j) &=&\sum_{j=1}^m \sum_{B\in \B_j} (v_{j-1}(B)-v_j(B))=\sum_{B\in \B_j}\sum_{j=1}^m  (v_{j-1}(B)-v_j(B))  \\
&\leq & \sum_{B\in \B}\sum_{j=1}^m (v_{j-1}(B)-v_j(B))\leq \sum_{B\in \B} u(B)=\sum_{i\in \I_{k}} c_i.
\end{eqnarray*}
Choose $j_0\in \{1,\ldots, m\}$ so that $\#(\B_{j_0})$ is the smallest. Then $$\#(\B_{j_0}) \leq \frac{1}{m}\sum_{i\in \I_{k}} c_i
\leq \frac{1}{t}\sum_{i\in \I_{k}} c_i.$$ Hence  $\J_{k}:=\{i\in \I_k:\; B_i\in \B_{j_0}\}$ is desired.

{\sl Step 2. There exists $\I'\subset \I$ with $\#(\I')\leq \frac{1}{t}\sum_{i\in \I}c_i$ so that $Z\subseteq \bigcup_{i\in \I'} 6B_i$.
 }

 Since $Z_{k}\uparrow Z$, $Z_{k}\neq \emptyset$ when $k$ is large enough. Let $\J_{k}$ be  constructed as in Step 1. Then $\J_{k}\neq \emptyset$ when $k$ is large enough.
 Define $G_{k}=\{x_i:\; i\in \J_{k}\}$.
 Then
 $$
\#(G_k)=\# (\J_k)\leq  \frac{1}{t}\sum_{i\in \I_{k}}c_i\leq \frac{1}{t}\sum_{i\in \I} c_i.
 $$

 Since the space of non-empty compact subsets of $X$ is compact with respect to the Hausdorff distance (cf. Federer \cite[2.10.21]{Fed69}),  there is a subsequence $(k_j)$ of natural numbers and a non-empty compact set $G\subset X$ such that $G_{k_j}$ converges to $G$ in  the Hausdorff distance as $j\to \infty$. As any two different points in $G_{k}$ have a distance  not less than $\epsilon$, so do the points in $G$. Thus $G$ is a finite set, moreover, $\#(G_{k_j})=\#(G)$ when $j$ is large enough.  Hence
 $$
 \bigcup_{x\in G} B(x,5.5\epsilon) \supseteq  \bigcup_{x\in G_{k_j}} B(x,5\epsilon)=\bigcup_{i \in \J_{k_j}} 5 B_i\supseteq Z_{k_j}
 $$
 when $j$ is large enough, and thus $\bigcup_{x\in G} B(x,5.5\epsilon)\supseteq Z$. On the other hand, when $j$ is large enough, we have
 $$
  \bigcup_{x'\in G_{k_j}} B(x',6\epsilon) \supseteq  \bigcup_{x\in G} B(x,5.5\epsilon),
 $$
 hence we have $\bigcup_{x'\in G_{k_j}} B(x',6\epsilon)\supseteq Z$, with $\#(G_{k_j})\leq \frac{1}{t}\sum_{i\in \I}c_i.$
\end{proof}

\bigskip
Return back to the metric spaces $(X_i, d_i)$ and TDS's $(X_i, T_i)$, $i=1,\ldots, k$. For $n\in \N$, define
a metric $d^\ba_n$ on $X_1$ by
$$
d^\ba_n(x,y)=\sup\left\{ d_i(T_i^j\tau_{i-1} x, T_i^j \tau_{i-1} y):  \; 1\leq i\leq k, \;0\leq j\leq \lceil(a_1+\ldots +a_i)n\rceil-1\right\}.
$$

\begin{lem}
\label{lem-51}
  Let $\epsilon>0$. Then there exist $\gamma>0$  such that for any $n\in  \N$, $X_1$ can be covered by no more than $\exp(n\gamma)$ balls of radius $\epsilon$ in  metric $d^\ba_n$.
\end{lem}
\begin{proof}
By compactness, for each $1\leq i\leq k$, we can find a finite open cover $\alpha_i$ of $X_i$ with $\mbox{diam}(\alpha_i)<\epsilon$ (in metric $d_1$). Let $n>0$.
Define
$$
\beta=\bigvee_{i=1}^k \left(
\bigvee_{j=0}^{\lceil(a_1+\cdots+a_i)n\rceil-1}T_1^{-j}\tau_{i-1}^{-1}\alpha_i\right).
$$
Then $\beta$ is an open cover of $X_1$  with diameter less than $\epsilon$ (with respect to the metric $d^\ba_n$).  Hence $X_1$ can be covered by at most $\#(\beta)$ many balls of radius $\epsilon$ in metric $d^\ba_n$.
Let $\gamma>0$ so that $\exp(\gamma)=\prod_{i=1}^k(\#(\alpha_i))^{a_1+\cdots+a_i+1}$. Then
$$
\#(\beta)\leq \prod_{i=1}^k(\#(\alpha_i))^{\lceil(a_1+\cdots+a_i)n\rceil}\leq \exp(n\gamma),
$$
which implies the result of the lemma.
\end{proof}

\begin{proof}[Proof of Proposition \ref{pro-2.1}] Let $Z\subseteq X_1$, $s\geq 0$, $\epsilon, \delta>0$. Taking  $f=\chi_Z$ and $c_i\equiv 1$ in the definition \eqref{e-19}, we see that
 $\W^{\ba,s}_{\Phi, N,\epsilon}(Z)\leq \Lambda^{\ba, s}_{\Phi, N,\epsilon}(Z)$ for each $N\in \N$. In the following, we prove that $\Lambda^{\ba, s+\delta}_{\Phi, N,6\epsilon}(Z)\leq \W^{\ba, s}_{\Phi, N,\epsilon}(Z)$
 when $N$ is large enough.

 Let $\gamma>0$ be given as in Lemma \ref{lem-51}.
 Assume that $N\geq 2$ such that
 \begin{equation}
 \label{e-gamma}
 n^2(n+1)e^{\gamma-n\delta}\le 1,\quad \forall\; n\geq N.
 \end{equation}
 Let $\{(n_i, A_i, c_i)\}_{i\in \mathcal I}$ be a  family so that $\I\subseteq \N$, $A_i\in {\mathcal T}^\ba_{n_i,\epsilon} $,  $0<c_i<\infty$,
  $n_i\geq N$ and
 \begin{equation} \label{e-gez}
  \sum_{i\in \I}c_i\chi_{A_i}\geq \chi_Z.
 \end{equation}
 We show below that
 \begin{equation}
 \label{e-key}
 \Lambda^{\ba, s+\delta}_{\Phi, N,6\epsilon}(Z)\leq \sum_{i\in \I}c_i\exp\left(-n_is+\frac{1}{a_1}\sup_{x\in A_j}\log \phi_{\lceil a_1n_j\rceil}(x)\right) ,
 \end{equation}
 which implies  $\Lambda^{\ba, s+\delta}_{\Phi, N,6\epsilon}(Z)\leq \W^{\ba,s}_{\Phi, N,\epsilon}(Z)$.

To prove \eqref{e-key}, we write   $\I_n=\{i\in \I:\; n_i=n\}$,
$$
g_n(x)=(\phi_{\lceil a_1n\rceil}(x))^{1/a_1},\quad g_n(E)=\sup_{x\in E} g_n(x)
$$
for $n\in \N$,  $x\in X_1$ and  $E\subseteq X_1$. Moreover set
 \begin{equation*}
 \begin{split}
  Z_{n,t}&=\Big\{x\in Z:\; \sum_{i\in \I_{n}}c_i\chi_{A_i}(x)>t\Big\}.
    \end{split}
 \end{equation*}
We claim that
 \begin{equation}
 \label{e-2.10}
\Lambda^{\ba,s+\delta}_{\Phi, N,\epsilon}(Z_{n,t})\leq \frac{1}{t n^2}\sum_{i\in \I_n}c_i \exp(-ns) g_n(A_i),\quad  \forall \; n\geq N, \;0<t<1.
 \end{equation}

 To prove the claim, assume that  $n\geq N$ and $0<t<1$. Set  $D=\frac{1}{n}\log g_n(Z_{n,t})$.
For $\ell=1,\ldots, n$ and  $i\in \I_n$,  write
$$
Z_{n,t}^\ell=\left\{x\in Z_{n,t}: \frac{1}{n}\log g_n(x)\in \Big(D-\frac{\gamma \ell}{n}, D-\frac{\gamma (\ell-1)}{n}\Big]\right\},\quad A_{i,\ell}:=A_i\cap Z_{n,t}^\ell,
$$
and
$$Z_{n,t}^0=\left\{x\in Z_{n,t}: \frac{1}{n}\log g_n(x)\leq  D-\gamma\right\},\quad A_{i,0}=A_i\cap Z_{n,t}^0.
$$
For $\ell=0, 1,\ldots, n$, write $\I_{n,\ell}=\{i\in \I_n:\; A_{i,\ell}\neq \emptyset\}$; then
 $$Z_{n, t}^{\ell}= \Big\{x\in X_1:\; \sum_{i\in \I_{n,\ell}}c_i\chi_{A_{i,\ell}}(x)>t\Big\}.$$
Hence by Lemma \ref{lem-2.0}, $Z_{n, t}^{\ell}$ can be covered by no more than $\frac{1}{t}\sum_{i\in \I_{n,\ell}} c_i$ balls with center in
 $\bigcup_{i\in \I_n}A_{i,\ell}$ and  radius $6\epsilon$ (in   metric $d_n^\ba$). It follows that for $\ell=1,\ldots, n$,
\begin{equation}
\label{e-im}
\begin{split}
\Lambda^{\ba,s+\delta}_{\Phi, N,6\epsilon}(Z_{n, t}^{\ell})&\leq e^{-n(s+\delta)} (\frac{1}{t}\sum_{i\in \I_{n,\ell}} c_i)   g_n(Z_{n,t}^\ell)\leq e^{-n(s+\delta)} e^{\gamma} \frac{1}{t}\sum_{i\in \I_{n,\ell}} c_i   g_n(A_{i,\ell})\\
&\leq e^{\gamma-n\delta}  \frac{1}{t}\sum_{i\in \I_{n}} c_i e^{-ns}   g_n(A_{i}).\\
\end{split}
\end{equation}
We still need to estimate $\Lambda^{\ba,s+\delta}_{\Phi, N,6\epsilon}(Z_{n, t}^{0})$. By Lemma \ref{lem-51},  $X_1$ (and thus $Z_{n, t}^{0}$) can be covered by no more than
$\exp(n\gamma)$ balls  of radius $6\epsilon$ (in  metric $d_n^\ba$).  Hence
\begin{equation}
\label{e-im1}
\begin{split}
\Lambda^{\ba,s+\delta}_{\Phi, N,6\epsilon}(Z_{n, t}^{0})&\leq \exp(n\gamma) e^{-n(s+\delta)}g_n(Z_{n, t}^{0})\leq  \exp(n\gamma) e^{-n(s+\delta)} \exp (n (D-\gamma))\\
&\leq e^{-n(s+\delta)} \exp (nD)\leq  e^{-n\delta} \frac{1}{t}\sum_{i\in \I_{n}} c_i e^{-ns}  g_n(A_{i}),
\end{split}
\end{equation}
where the last inequality uses the following arguments: since $\exp (nD)=g_n(Z_{n,t})$, for any $u<\exp (nD)$, there exists $x\in Z_{n,t}$ so that
$g_n(x)\geq u$; however since $x\in Z_{n,t}$ we have $\sum_{i\in \I_n: \; A_i\ni x}c_i\geq t$, which implies
$$
\frac{1}{t}\sum_{i\in \I_{n}} c_i   g_n(A_{i})\geq \frac{1}{t}\sum_{i\in \I_{n}: A_i\ni x } c_i   g_n(A_{i})\geq \frac{1}{t}\sum_{i\in \I_{n}: A_i\ni x } c_i  u\geq u.
$$

 Combining \eqref{e-im}-\eqref{e-im1}, we have
 \begin{equation}
 \label{e-im2}
\begin{split}
 \Lambda^{\ba,s+\delta}_{\Phi, N,6\epsilon}(Z_{n, t})\leq \sum_{\ell=0}^n \Lambda^{\ba,s+\delta}_{\Phi, N,6\epsilon}(Z_{n, t}^{\ell})&\leq
 (n+1)e^{\gamma-n\delta} \frac{1}{t}\sum_{i\in \I_{n}} c_i e^{-ns}   g_n(A_{i})\\
 &\leq \frac{1}{n^2t}\sum_{i\in \I_{n}} c_i e^{-ns}   g_n(A_{i}),
 \end{split}
\end{equation}
where in the last inequality we use \eqref{e-gamma}. This finishes the proof of \eqref{e-2.10}.

 To complete the proof of Proposition  \ref{pro-2.1},  notice that $\sum_{n=N}^\infty n^{-2}\leq \sum_{n=2}^\infty n^{-2}\leq 1$; hence if $x\not\in
 \bigcup_{n\geq N} Z_{n, n^{-2}t}$, then
 $$
 \sum_{i\in \I}c_i\chi_{A_i}(x)=\sum_{i\in \bigcup_{n=N}^\infty\I_n}c_i\chi_{A_i}(x)\leq \sum_{n=N}^\infty\sum_{i\in \I_n}c_i\chi_{A_i}(x)\leq \sum_{n=N}^\infty n^{-2}t\leq t<1,
 $$
thus  $x\not\in Z$ by \eqref{e-gez}. Therefore $Z\subseteq \bigcup_{n\geq N} Z_{n, n^{-2}t}$. By \eqref{e-im2},
 $$
 \Lambda^{\ba,s+\delta}_{\Phi, N,6\epsilon}(Z)\leq \sum_{n=N}^\infty \Lambda^{\ba,s+\delta}_{\Phi, N,6\epsilon}(Z_{n, n^{-2}t})\leq
 \frac{1}{t}\sum_{n=N}^\infty\sum_{i\in \I_{n}} c_i e^{-ns}   g_n(A_{i})\leq \frac{1}{t}\sum_{i\in \I} c_i e^{-n_is}   g_{n_i}(A_{i}).
  $$
 Letting $t\uparrow 1$, we have
 $$
\Lambda^{\ba,s+\delta}_{\Phi, N,6\epsilon}(Z)\leq \sum_{i\in \I} c_i e^{-n_is}   g_{n_i}(A_{i}),
 $$
 that is, \eqref{e-key} holds.
 This finishes the proof of  Proposition  \ref{pro-2.1}.
 \end{proof}

\subsection{Proof of Lemma \ref{lem-Frost}}
\label{S-3}

It is easy to see that Lemma \ref{lem-Frost} follows directly from Proposition  \ref{pro-2.1} and the following lemma.
\begin{lem} \label{pro-3.1}  Let $s\geq 0$, $N\in \N$ and $\epsilon>0$. Suppose that
 $c:=\W^{\ba,s}_{\Phi, N,\epsilon}(X_1)>0$. Then there is a Borel probability measure $\mu$ on $X_1$ such that  for any $n\geq N$, $x\in X_1$,  and any compact
 $K\subset B^\ba_n(x,\epsilon)$,
   $$
\mu(K)\leq \frac{1}{c}e^{-ns} g_n(K),
$$
where $$
g_n(z)=(\phi_{\lceil a_1n\rceil}(z))^{1/a_1},\quad g_n(K)=\sup_{z\in K} g_n(z).
$$
\end{lem}
\begin{proof} Here we adopt the idea employed by  Howroyd  in his proof of the Frostman lemma in compact metric spaces (cf. \cite[Theorem 2]{How95}). Clearly $c<\infty$. We define a function $p$ on the space $C(X_1)$ of continuous real-valued functions on $X_1$ by
\begin{equation*}
\label{zx-eq}
p(f)=(1/c)\W^{\ba,s}_{\Phi, N,\epsilon}(f).
\end{equation*}

Let ${\bf 1}\in C(X_1)$ denote the constant function ${\bf 1}(x)\equiv 1$. It is easy to verify that

\begin{enumerate}

\item $p(f+g)\le p(f)+p(g)$ for any $f,g\in C(X_1)$.

\item $p(tf)=tp(f)$ for any $t\ge 0$ and $f\in C(X_1)$.

\item $p({\bf 1})=1$, $0\leq p(f)\leq \|f\|_\infty$ for any $f\in C(X_1)$,  and
$p(g)=0$ for $g\in C(X_1)$ with $g\le 0$.
\end{enumerate}
By
the Hahn-Banach theorem, we can extend  the linear functional $t\mapsto t p({\bf 1})$, $t\in \R$, from the subspace of the constant functions to a linear functional $L:\;
C(X_1)\to \R$ satisfying
$$L({\bf 1})=p({\bf 1})=1 \text{ and }-p(-f)\le L(f)\le p(f)    \text{ for any }f\in C(X_1).$$
If $f\in C(X_1)$ with $f\ge 0$, then $p(-f)=0$ and so $L(f)\ge 0$.
Hence combining the fact $L({\bf 1})=1$, we can use the Riesz
representation theorem to find a Borel probability measure $\mu$ on
$X_1$ such that $L(f)=\int f d\mu$ for $f\in C(X_1)$.

Now let $x\in X_1$ and $n\geq N$. Suppose that  $K$ is a compact subset of $B^\ba_n(x,\epsilon)$.  Let $\delta>0$. Since $g_n$ is upper semi-continuous, there exists an open set $B^\ba_n(x,\epsilon)\supset V\supset K$ such that
$g_n(V)\leq g_n(K)+\delta$.

By the Uryson lemma,  there exists $f\in C(X_1)$ such that $0
\le f\le 1$, $f(y)=1$ for $y\in K$, and $f(y)=0$ for $y\in X_1\backslash V$.
Then $\mu(K)\le L(f)\le p(f)$. Since $f\leq \chi_{V}$ and $n\geq N$, we have $\W^{\ba, s}_{\Phi, N,\epsilon}( f)\leq e^{-ns}g_n(V)$ and thus
$p(f)\le \frac{1}{c} e^{-sn}g_n(V)$. Therefore $$\mu(K)\le \frac{1}{c}e^{-ns}g_n(V)\leq \frac{1}{c}e^{-ns} (g_n(K)+\delta).$$ Letting $\delta\to 0$, we have
$\mu(K)\le \frac{1}{c}e^{-ns}g_n(K).$ This completes the proof of the lemma.
\end{proof}

\section{The proof of Theorem \ref{thm-1.1}: Lower bound}
\label{s-4}

In this section, we prove the lower bound part of Theorem \ref{thm-1.1}. The following weighted version of Brin-Katok theorem plays a key role in our proof.
\begin{thm} \label{thm-4.1} For each ergodic measure $\mu\in \mathcal{M}(X_1,T_1)$, we have
$$\lim_{\epsilon \rightarrow 0} \liminf_{n\rightarrow +\infty} \frac{-\log \mu(B_n^{\bf a}(x,\epsilon))}{n}=
\lim_{\epsilon \rightarrow 0} \limsup_{n\rightarrow +\infty}
\frac{-\log \mu(B_n^{\bf a}(x,\epsilon))}{n}=h_\mu^{\bf a}(T_1)$$
for $\mu$-a.e. $x\in X_1$.
\end{thm}
We shall postpone the proof of Theorem \ref{thm-4.1} to  Appendix \ref{s-a}. In the following we prove the lower bound part of Theorem \ref{thm-1.1} for sub-additive potentials rather than additive potentials.
\begin{pro}
\label{thm-3.3}
Let  $\Phi=\{\log \phi_n\}_{n=1}^\infty$ be a sub-additive potential on $X_1$.
%Assume that $\htop(T_1)<\infty$.
 Then
$$
P^\ba (T_1,\Phi)\ge
\sup\left\{\Phi_*(\mu)+h^\ba_\mu(T_1):\; \mu\in \M(X_1,T_1), \Phi_*(\mu)\neq -\infty \right\}.
$$
\end{pro}
\begin{proof}   By Jacobs' theorem (cf. \cite[Theorem 8.4]{Wal82}) and Proposition A.1.(3) in \cite{FeHu10}, if $\mu=\int_{\E(X_1, T_1)} m \;d\tau(m)$ is the ergodic decomposition of an element $\mu$ in $\M(X_1, T_1)$, then
$$h^\ba_\mu(T_1)=\int_{\E(X_1, T_1)} h^\ba_m(T_1) \;d\tau(m),\quad  \Phi_*(\mu)=\int_{\E(X_1, T_1)} \Phi_*(m)\;d\tau(m).$$
 Hence to prove the proposition, it suffices to show that
\begin{equation}
\label{e-0.2}
P^\ba (T_1,\Phi)\geq \Phi_*(\mu)+\min\{\delta^{-1}, \;h^\ba_\mu(T_1)-\delta\}-\delta
\end{equation}
 for any $\delta>0$ and any ergodic $\mu\in \M(X_1,T_1)$ with $\Phi_*(\mu)\neq -\infty$.

 For this purpose, we fix  $\delta>0$ and an ergodic measure $\mu$ on $X_1$ with $\Phi_*(\mu)\neq -\infty$. Write $$H:=\min\{\delta^{-1},  \;h^\ba_\mu(T_1)-\delta\}.$$
By Theorem  \ref{thm-4.1}, we can choose $\epsilon>0$ so that
\begin{equation}
\label{e-2.1}
\liminf_{n\to \infty}\frac{-\log \mu(B^\ba_n(x,\epsilon))}{n}>H \quad \mbox{ for  $\mu$-a.e. }x\in X_1.
\end{equation}
Since $\Phi$ is  sub-additive,  by Kingman's subadditive ergodic theorem (cf. \cite[p. 231]{Wal82} and \cite[Proposition A.1.]{FeHu10}), we have
$$\lim_{n\to \infty} \frac{1}{n} \log \phi_n(x)=\Phi_*(\mu)$$ for $\mu$-a.e. $x\in X_1$. Hence
 there exists a large $N\in \N$ and a Borel set $E_N\subset X_1$ with $\mu(E_N)>1/2$ such that for any  $x\in E_N$ and $n\geq N$,
\begin{equation}
\label{e-2.2}
\mu(B^\ba_n(x,\epsilon))<\exp(-nH),\quad \log \phi_{\lceil a_1n\rceil}(x)\geq  a_1n\Phi_*(\mu)-a_1n\delta.
\end{equation}

Now assume that $\Gamma
=\{(n_j, A_j)\}_i$ is a countable collection so that $n_j\geq N$, $A_j\in {\mathcal T}^\ba_{n_j,\epsilon/2}$ (cf. \eqref{e-0.1} for the definition) and $\bigcup_jA_j=X_1$.  By definition, for each $j$, there exists  $x_j\in X$ so that $A_j\subseteq B^\ba_{n_j}(x_j,\epsilon/2)$.
  Set $$\I:=\{j:\; A_j\cap E_N\neq \emptyset\}.$$
   For $j\in \I$, pick $y_j\in  A_j\cap E_N$; then we have
$$
A_j\subseteq B_{n_j}^\ba(x_j,\epsilon/2)\subseteq  B_{n_j}^\ba(y_j,\epsilon)
$$
and thus
$$
 \mu(A_j)\leq \mu(B_{n_j}^\ba(y_j,\epsilon))\leq \exp(-n_jH);
$$
moreover,
$$
\frac{1}{a_1} \sup_{x\in A_j} \log \phi_{\lceil a_1n_j \rceil}(x) \geq \frac{1}{a_1}\log  \phi_{\lceil a_1n_j \rceil}(y_j)\geq n_j\Phi_*(\mu)-n_j\delta.
$$
Set
$s:=\Phi_*(\mu)+H-\delta$. Then for any $j\in \I$,
$$
\exp\left(-sn_j+\frac{1}{a_1}\sup_{x\in A_j}\phi_{\lceil a_1n_j\rceil}(x)\right )
\geq {\mu(A_j)} \exp \left(n_j \left(-s+\Phi_*(\mu)+H-\delta\right)\right)=\mu(A_j).
$$
Summing over $j\in \I$, we have
\begin{equation*}
\begin{split}
\sum_{j\in \I}\exp \left(-sn_j+\frac{1}{a_1} \sup_{x\in A_j} \phi_{\lceil a_1n_j \rceil}(x)\right)&\geq \sum_{j\in \I}\mu(A_j)
\geq \mu\left(\bigcup_{j\in \I} A_j\right)
\geq \mu(E_N)\geq 1/2.
\end{split}
\end{equation*}
It follows that $\Lambda_{\Phi, \epsilon}^{\ba, s}(X_1)\geq \Lambda_{\Phi, N, \epsilon}^{\ba,  s}(X_1)\geq 1/2$, and thus
$$P^\ba(T_1,\Phi)\geq P^\ba (T_1,\Phi, X_1,\epsilon/2)\geq s= \Phi_*(\mu)+\min\{\delta^{-1}, \;h^\ba_\mu(T_1)-\delta\}-\delta,$$
 as desired.
\end{proof}

\section{The proof of Theorem \ref{thm-1.1}: upper bound}
\label{s-5}
In this section, we prove the upper bound in Theorem \ref{thm-1.1}, that is, for any $f\in C(X_1)$ and $\delta>0$, there exists $\mu\in \M(X_1, T_1)$ such that
$$
P^\ba(T_1, f)\leq h^\ba_\mu(T_1)+\int_{X_1} f d\mu+\delta.
$$

Before proving the above result, we first give some lemmas.

\begin{lem}\label{lem-en.0}
Let $(X, T)$ be a TDS and $\mu\in \M(X)$. Let $\alpha=\{A_1,\ldots, A_M\}$ be a  Borel partition of $X$ with cardinality $M$.
Write  for brevity
$$h(n):=H_{\frac{1}{n}\sum_{i=0}^{n-1}\mu\circ T^{-i}}(\alpha), \quad h(n,m):=H_{\frac{1}{m}\sum_{i=n}^{m+n-1}\mu\circ T^{-i}}(\alpha).
$$
for  $n, m\in \mathbb{N}$.
Then
\begin{itemize}
\item[(i)] $h(n)\leq \log M$ and $h(n,m)\leq \log M$ for $n,m\in \N$.
\item[(ii)] $|h(n+1)-h(n)|\leq \frac{1}{n+1}\log \left(3M^2(n+1)\right)$ for all $n\in \mathbb{N}$.
\item[(iii)] $\left|h(n+m)-\frac{n}{n+m}h(n)-\frac{m}{n+m}h(n,m)\right|\leq \log 2$ for all $n, m\in \N$.
\end{itemize}
\end{lem}
\begin{proof} (i) is obvious. Now we turn to the proof of (ii). It is well known (see e.g. \cite[Theorem 8.1]{Wal82} and the proof therein) that for any  finite Borel partition $\beta$ of $X$,
$\nu_1,\nu_2\in \M(X)$ and $p\in [0,1]$,
\begin{equation} \label{eq-en-es}
\begin{aligned}
0&\le H_{p\nu_1+(1-p)\nu_2}(\beta)-pH_{\nu_1}(\beta)-(1-p)H_{\nu_2}(\beta)\\
&\le -(p\log p+(1-p)\log (1-p))\\
&\le \log 2.
\end{aligned}
\end{equation}
Let $n\in \mathbb{N}$. Applying \eqref{eq-en-es} and (i), we have
\begin{align*}
| h & (n+1) -h(n)|\\
&=\Big| h(n+1)-\frac{n}{n+1}h(n)-\frac{1}{n+1}H_{\mu\circ T^{-n}}(\alpha)-\frac{1}{n+1}h(n)+\frac{1}{n+1}H_{\mu\circ T^{-n}}(\alpha)\Big|\\
%&\le | h(n+1)-\frac{n}{n+1}h(n)-\frac{1}{n+1}H_{\mu\circ T^{-n}}(\alpha)|+\frac{1}{n+1}h(n)+\frac{1}{n+1}H_{\mu\circ T^{-n}}(\alpha)\\
&\le \Big| h(n+1)-\frac{n}{n+1}h(n)-\frac{1}{n+1}H_{\mu\circ T^{-n}}(\alpha)\Big|+\frac{2}{n+1}\log M\\
&\le -\frac{n}{n+1}\log \frac{n}{n+1}-\frac{1}{n+1}\log \frac{1}{n+1}+\frac{2}{n+1}\log M\\
&\le \frac{1}{n+1}\log \big(3 M^2(n+1)\big),
\end{align*}
where we use the fact $(1+1/n)^n<e<3$ in the last inequality.
This proves (ii).

Finally, since
$$\frac{1}{n+m}\sum_{i=0}^{n+m-1}\mu\circ T^{-i}=\frac{n}{n+m}\left (\frac{1}{n}\sum_{i=0}^{n-1}\mu\circ T^{-i}\right)+\frac{m}{n+m}\left(\frac{1}{m}\sum_{i=n}^{n+m-1}\mu\circ T^{-i}\right)$$
for $n,m\in \mathbb{N}$, (iii) follows from \eqref{eq-en-es}.
\end{proof}

\begin{lem}
\label{lem-en.1}
Let $(X, T)$ be a TDS and $\mu\in \M(X)$. For $\epsilon>0$ and $\ell,M\in \N$, let $H_\bullet(\epsilon, M; \ell)$ be defined as in \eqref{e-HEML}. Then the following statements hold.
\begin{enumerate}
\item For all $n\in \mathbb{N}$,
\begin{align*}
\Big|H_{\frac{1}{n} \sum_{i=0}^{n-1} \mu\circ T^{-i}} & (\epsilon, M;\ell)-H_{\frac{1}{n+1}\sum_{i=0}^{n}\mu\circ T^{-i}}(\epsilon, M;\ell) \Big|\\
& \le \frac{1}{\ell(n+1)}\log \big(3M^{2\ell}(n+1)\big).
\end{align*}
\item  For all $n, m\in \N$,
\begin{equation}
\label{e-inequality}
\begin{aligned}
 \frac{n}{n+m}&  H_{\frac{1}{n}\sum_{i=0}^{n-1}\mu\circ T^{-i}}(\epsilon, M;\ell)
  +\frac{m}{n+m}H_{\frac{1}{m}\sum_{i=n}^{n+m-1}\mu\circ T^{-i}}(\epsilon, M;\ell)\\
 &\leq  H_{\frac{1}{n+m}\sum_{i=0}^{m+n-1}\mu\circ T^{-i}}(\epsilon, M;\ell)+\frac{\log 2}{\ell}.
\end{aligned}
\end{equation}
\end{enumerate}
\end{lem}
\begin{proof} The statements directly follow  from  the definition of $H_{\bullet}(\epsilon, M;\ell)$ and Lemma \ref{lem-en.0}.
\end{proof}

\begin{lem}[Lemma 2.4 of \cite{CFH08}]
\label{lem-6.1} Let $\nu\in \M(X)$ and $M\in \N$. Suppose $\xi=\{A_1,\ldots,A_j\}$
is a Borel  partition of $X$ with $j\leq M$. Then for any positive integers
$n,\ell$ with $n\geq 2\ell$, we have
$$
\frac 1n H_\nu\left(\bigvee_{i=0}^{n-1}T^{-i}\xi\right)\leq \frac
1\ell
H_{\nu_n}\left(\bigvee_{i=0}^{\ell-1}T^{-i}\xi\right)+\frac{2\ell}{n}\log
M,
$$
where $\nu_n=\frac{1}{n}\sum_{i=0}^{n-1}\nu\circ T^{-i}$.
\end{lem}

The following lemma is a slight variant of  \cite[Lemma 4.1] {KePe96} by Kenyon and Peres.

\begin{lem}
\label{lem-KP}
Let $p\in \N$. Let $u_j: \N\to \R$ ($j=1,\ldots, p$) be bounded  functions with
$$
\lim_{n\to \infty} |u_j(n+1)-u_j(n)|=0.
$$
Then for any positive numbers $c_1,\ldots, c_p$ and $r_1,\ldots, r_p$,
$$
\limsup_{n\to +\infty}\sum_{j=1}^p (u_j(\lceil c_jn\rceil)-u_j(\lceil r_jn \rceil))\geq 0.
$$
\end{lem}
\begin{proof}
 For the convenience of reader, we give a proof by adapting the argument of Kenyon and Peres in \cite{KePe96}.

For $j=1,\ldots, p$, extend  $u_j$ in a piecewise linear fashion to a bounded continuous function on $[1,+\infty)$. Then for each $1\leq j\leq p$,
\begin{equation}\label{eq-KP}
\lim_{t\rightarrow +\infty} \sup\left\{ |u_j(x)-u_j(y)|:x,y\ge t,\; |x-y|\leq \max_{1\leq i\leq p}\max\{c_i,r_i, 1\}\right\}=0.
\end{equation}
 Take a positive number $M$ so that
 \begin{equation}
 \label{e-M}
  M>\max_{1\le j\le p}\{ |\log c_j|+|\log r_j|+1\}.
  \end{equation}
  Then for every $w>M$,
\begin{align*}
 & \Big|  \int_M^w  \sum_{j=1}^p \big( u_j(e^{x+\log c_j})-u_j(e^{x+\log r_j})\big)dx \Big| \\
 &\quad =\left| \sum_{j=1}^p \left[ \int_{M+\log c_j}^{w+\log c_j} u_j(e^x)dx - \int_{M+\log r_j}^{w+\log r_j} u_j(e^x)dx  \right]  \right|\\
 &\quad \leq \sum_{j=1}^p \left|  \int_{M+\log c_j}^{w+\log c_j} u_j(e^x)dx - \int_{M+\log r_j}^{w+\log r_j} u_j(e^x)dx    \right|\\
&\quad = \sum_{j=1}^p \Big|\int^{w+\log c_j}_{w+\log r_j} u_j(e^x) dx-\int^{M+\log c_j}_{M+\log r_j} u_j(e^x) dx\Big|,
\end{align*}
Since each $u_j$ is bounded, the sum in the right-hand side of the last `$=$' above is uniformly bounded. It follows that
$$\limsup_{x\rightarrow +\infty}\sum_{j=1}^p \big( u_j(e^{x+\log c_i})-u_j(e^{x+\log r_j})\big)\ge 0.$$
Setting $t=e^x$, one has
$$\limsup_{t\rightarrow +\infty}\sum_{j=1}^p(u_j( c_it)-u_j(r_j t ))\ge 0.$$
Combining the above inequality with \eqref{eq-KP}, we  have
\begin{align*}
\limsup_{n\rightarrow +\infty}& \sum_{j=1}^p(u_j(\lceil c_jn\rceil)-u_j(\lceil r_j n\rceil)\\
&=\limsup_{n\rightarrow +\infty}\sum_{j=1}^p(u_j(c_jn)-u_j(r_j n))
\\
&= \limsup_{t\rightarrow +\infty}\sum_{j=1}^p(u_j( c_j\lceil t\rceil )-u_j(r_j \lceil t \rceil ))\\
&= \limsup_{t\rightarrow +\infty}\sum_{j=1}^p(u_j( c_jt)-u_j(r_j t ))\ge 0,
\end{align*}
which completes the proof of the lemma.
\end{proof}

\begin{proof}[Proof of Theorem \ref{thm-1.1}: upper bound] Suppose that $P^{\ba}(T_1, f)>0$.
 Fix $0<s<s'<P^\ba (T_1,f)$. Let $\Phi=\{\log \phi_n\}_{n=1}^\infty$ be the additive potential generated by $f$, that is, $\phi_n(x)=\exp(S_nf(x))$ where $S_nf(x):=\sum_{i=0}^{n-1}f(T_1^ix)$.  Take $\epsilon_0>0$ such that
 \begin{equation}
 \label{star}
 \sup\{|f(x)-f(y)|:\; x,y\in X_1,\; d_1(x,y)\leq \epsilon_0\}<(s'-s)a_1/(1+a_1).
 \end{equation}
  By Lemma \ref{lem-Frost}, there exist $\nu\in \M(X_1)$, $\epsilon\in (0,\epsilon_0)$,  and $N\in \N$ such that
\begin{equation}
\label{e-6.1}
\begin{split}
\nu(B^\ba_n(x,\epsilon))&\leq  \sup_{y\in B^\ba_n(x,\epsilon)}\exp\left(-s'n + \frac{1}{a_1}S_{\lceil a_1n\rceil}f(y)\right)\\
&\leq
\exp\left(-sn + \frac{1}{a_1}S_{\lceil a_1n\rceil}f(x)\right)
\end{split}
\end{equation}
for any $n\geq N$ and $x\in X_1$, where in the last inequality we use \eqref{star}.

By continuity,  there exists $\tau\in (0,  \epsilon)$ such that for any $1\le i<j\le k$, if $x_i,y_i\in X_i$ satisfy $d_i(x_i,y_i)<\tau$,
then $$d_j(\pi_{j-1}\circ \cdots\circ \pi_i(x_i),\pi_{j-1}\circ \cdots\circ \pi_i(y_i))<\epsilon.$$  Take $M_0\in \mathbb{N}$ with
$\mathcal{P}_{X_i}(\tau,M_0)\neq \emptyset$ for $i=1,\ldots,k$, where $\mathcal{P}_{X_i}(\tau,M_0)$ is defined as in \eqref{e-2014-6}.    Now fix $M\in \N$ with $M\ge M_0$.
Let $\alpha_i\in \mathcal{P}_{X_i}(\tau,M)$ for $i=1,\ldots, k$. Set $\beta_i=\tau_{i-1}^{-1}\alpha_i$ and write for brevity that
$$t_0(n)=0, \quad t_i(n)=\lceil (a_1+\ldots +a_{i})n\rceil$$
for $n\in \N$ and $i=1,\ldots, k$. Then
for any $n\in \N$ and $x\in X_1$, we have
\begin{equation}
\label{e-relation}
\bigvee_{i=1}^k\bigvee_{j=t_{i-1}(n)}^{t_i(n)-1}T^{-j}_1\beta_i(x)\subseteq B^\ba_n(x,\epsilon).
\end{equation}
Now assume that $n\geq N$. By \eqref{e-6.1} and \eqref{e-relation},
\begin{equation}
\label{e-6.2}
\nu\Big(\bigvee_{i=1}^k\bigvee_{j=t_{i-1}(n)}^{t_i(n)-1}T^{-j}_1\beta_i(x)\Big)\leq  \exp\left(-sn + \frac{1}{a_1}S_{\lceil a_1n\rceil}f(x)\right)
\end{equation}
for any $x\in X_1$.
It follows that
\begin{align*}
H_{\nu}\Big(\bigvee_{i=1}^k\bigvee_{j=t_{i-1}(n)}^{t_i(n)-1}T^{-j}\beta_i\Big) &= -\int \log \nu\Big(\bigvee_{i=1}^k\bigvee_{j=t_{i-1}(n)}^{t_i(n)-1}T^{-j}_1\beta_i(x)\Big) d\nu(x)\\
& \geq sn  -\int \frac{1}{a_1}S_{\lceil a_1n\rceil}f(x) d\nu(x).
\end{align*}
Hence
\begin{equation}
\label{e-6.2}
\sum_{i=1}^k H_{\nu}\Big(\bigvee_{j=t_{i-1}(n)}^{t_i(n)-1}T^{-j}_1\beta_i\Big)\geq sn  -\int \frac{1}{a_1}S_{\lceil a_1n\rceil}f(x) d\nu(x).
\end{equation}

Now fix $\ell\in \N$. By Lemma \ref{lem-6.1},  the left-hand side of  \eqref{e-6.2} is bounded from above by
$$
\sum_{i=1}^k \frac{t_i(n)-t_{i-1}(n)}{\ell}
H_{w_{i,n}}\Big(\bigvee_{j=0}^{\ell-1}T^{-j}_1\beta_i\Big)+2k\ell\log M,
$$
where $$
w_{i,n}:=\frac{ \sum_{j=t_{i-1}(n)}^{t_i(n)-1} \nu\circ  T^{-j}_1} {t_{i}(n)-t_{i-1}(n)}.
$$
Hence by \eqref{e-6.2} and the definition of $H_\bullet(\tau, M;\ell)$ (cf. \eqref{e-HEML}), we have
\begin{equation}
\label{e-2014'}
\begin{aligned}
\sum_{i=1}^k & (t_i(n)-t_{i-1}(n))
H_{w_{i,n}\circ \tau_{i-1}^{-1}}(\tau, M; \ell)\\
& \geq  sn  -\frac{\lceil a_1n\rceil}{a_1}\int f d w_{1,n}-2k\ell\log M.
\end{aligned}
\end{equation}

Define $\nu_{m}=\frac{\sum_{j=0}^{m-1}\nu\circ T^{-j}_1}{m}$ for $m\in \mathbb{N}$. For $i=1,\ldots,k$, we have
$$
\nu_{m}\circ \tau^{-1}_{i-1}=\frac{\sum_{j=0}^{m-1}(\nu\circ \tau^{-1}_{i-1})\circ T^{-j}_i}{m}, \, w_{i,n}\circ \tau_{i-1}^{-1}=\frac{ \sum_{j=t_{i-1}(n)}^{t_i(n)-1} (\nu\circ \tau_{i-1}^{-1})\circ  T^{-j}_i} {t_{i}(n)-t_{i-1}(n)}
$$
 and
\begin{align}\label{e-2014-0}
\nu_{t_i(n)}\circ \tau^{-1}_{i-1}=\frac{t_{i-1}(n)}{t_i(n)}\nu_{t_{i-1}(n)}\circ \tau^{-1}_{i-1}+\frac{t_i(n)-t_{i-1}(n)}{t_i(n)}w_{i,n}\circ \tau_{i-1}^{-1}.
\end{align}

Applying  Lemma \ref{lem-en.1}(2) to the measure $\nu\circ \tau_{i-1}^{-1}$ (more precisely, in \eqref{e-inequality}, we replace the terms $T$, $\mu$, $n$, $m$
by $T_i$, $\nu\circ \tau_{i-1}^{-1}$, $t_{i-1}(n)$, $t_i(n)-t_{i-1}(n)$, respectively), we have
\begin{align*}
\frac{t_{i-1}(n)}{t_i(n)}  H_{\nu_{t_{i-1}(n)}\circ \tau_{i-1}^{-1}} & (\tau,M,\ell)+ \frac{t_i(n)-t_{i-1}(n)}{t_i(n)} H_{w_{i,n}\circ \tau_{i-1}^{-1}}(\tau, M; \ell)\\
&\le H_{\nu_{t_i(n)}\circ \tau_{i-1}^{-1}}(\tau, M;\ell)
+\frac{\log 2}{\ell}.
\end{align*}
That is,
\begin{align*}
t_i(n)H_{\nu_{t_i(n)}\circ \tau_{i-1}^{-1}} & (\tau, M;\ell)-t_{i-1}(n)H_{\nu_{t_{i-1}(n)}\circ \tau_{i-1}^{-1}}(\tau,M,\ell)\\
&\ge (t_i(n)-t_{i-1}(n))
H_{w_{i,n}\circ \tau_{i-1}^{-1}}(\tau, M; \ell)-\frac{t_i(n)\log 2}{\ell}.
\end{align*}
Combining the above inequality with \eqref{e-2014'}, we have
\begin{equation}
\label{e-2014}
\begin{aligned}
\Theta_n: = & \sum_{i=1}^k    \left(t_i(n)H_{\nu_{t_i(n)}\circ \tau_{i-1}^{-1}}(\tau, M;\ell)-t_{i-1}(n)H_{\nu_{t_{i-1}(n)}\circ \tau_{i-1}^{-1}}(\tau,M,\ell)\right)\\
\geq & sn  -\frac{t_1(n)}{a_1}\int f d \nu_{t_1(n)}-2k\ell\log M-\frac{k t_{k}(n)\log 2}{\ell}.
\end{aligned}
\end{equation}

Write $g_i(n):=H_{\nu_{n}\circ \tau_{i-1}^{-1}}(\tau, M;\ell)$. Then by Lemma \ref{lem-en.1}(1),
\begin{equation}
\label{e-2014-1}
|g_i(n)-g_i(n+1)|\leq \frac{1}{\ell(n+1)}\log \big( 3M^{2\ell}(n+1)\big).
\end{equation}
Set
$$\gamma(n):=\sum_{i=2}^k t_i(n) (g_i(t_i(n))-g_i(t_1(n))) -\sum_{i=2}^k t_{i-1}(n)(g_i(t_{i-1}(n))-g_i(t_1(n))).
$$
Then we have
$$
\Theta_n=\gamma(n)+ \sum_{i=1}^k (t_i(n)-t_{i-1}(n))g_i(t_1(n)),
$$
where $\Theta_n$ is defined as in \eqref{e-2014}. Hence by  \eqref{e-2014},
we have
\begin{equation}
\label{e-2014-2}
\begin{aligned}
 \sum_{i=1}^k & \frac{t_i(n)-t_{i-1}(n)}{n} g_i(t_1(n))+\frac{t_1(n)}{a_1n} \int f d \nu_{t_1(n)} \\
 & \geq -\frac{\gamma(n)}{n}+s-\frac{2k\ell\log M}{n}-\frac{k t_{k}(n)\log 2}{n\ell}.
\end{aligned}
\end{equation}

 Define
$$
w(n)=\sum_{i=2}^k (a_1+\cdots+a_{i-1})(g_i(t_{i-1}(n))-g_i(t_1(n)))-\sum_{i=2}^k (a_1+\cdots+ a_i) (g_i(t_i(n))-g_i(t_1(n))).
$$
Then we have $\limsup_{n\to \infty} w(n)\geq 0$ by applying Lemma \ref{lem-KP}, in which we take $p=2k-2$,
$$
u_j(n)=\left\{ \begin{array}{ll}
(a_1+\cdots+a_{j})g_{j+1}(n) & \mbox{ if } 1\leq j\leq k-1,\\
-(a_1+\cdots+a_{j-k+2})g_{j-k+2}(n) &  \mbox{ if }k \leq j\leq 2k-2,
\end{array}
\right.
$$
and
$$
c_j=\left\{ \begin{array}{ll}
a_j & \mbox{ if } 1\leq j\leq k-1,\\
a_{j-k+2} &  \mbox{ if }k \leq j\leq 2k-2,
\end{array}
\right.
$$
and $r_j=1$ for all $j$; the condition $\lim_{n\to \infty}|u_j(n+1)-u_j(n)|=0$ fulfils, thanks to \eqref{e-2014-1}.

Since  $g_i$'s are bounded functions, we have  $$\limsup_{n\to \infty}\frac{-\gamma(n)}{n}=\limsup_{n\to \infty}w(n)\geq 0.$$
Hence letting $n\to \infty$ in  \eqref{e-2014-2} and taking the upper limit, we obtain
\begin{equation}
\label{e-2014-3}
\limsup_{n\to \infty}\left( \sum_{i=1}^k a_i g_i(t_1(n))+ \int f dv_{t_1(n)}\right)  \geq  s-\frac{k (a_1+\cdots+a_k) \log 2}{\ell}.
\end{equation}
Take a subsequence $(n_j)$ of the natural numbers so that the left-hand side of \eqref{e-2014-3} equals
$$
\lim_{j\to \infty}\left( \sum_{i=1}^k a_i H_{\nu_{t_1(n_j)} \circ \tau_{i-1}^{-1}}(\tau, M;\ell) + \int f d\nu_{t_1(n_j)}\right)
$$
and moreover, $\nu_{t_1(n_j)}$ converges to an element $\lambda\in \M(X_1, T_1)$ in the weak* topology.
Since the map  $H_\bullet(\tau, M;\ell)$ is upper semi-continuous on $\M(X_1)$ (see Lemma \ref{usc}), we have
\begin{equation}
\label{e-2014-4}
\sum_{i=1}^k a_i H_{\lambda \circ \tau_{i-1}^{-1}}(\tau, M;\ell) + \int f d\lambda\geq s-\frac{k (a_1+\cdots+a_k) \log 2}{\ell}.
\end{equation}

Define $$\E:=\left\{(M,\ell,\delta): M,\ell\in \mathbb{N},\delta>0 \text{ with }M\ge M_0,  \ell \ge \frac{k (a_1+\cdots+a_k) \log 2}{\delta}\right\}$$
and $$\Omega_{M,\ell, \delta}:=\left\{\eta\in \M(X_1, T_1):  H^\ba_\eta(\tau,M;\ell)+\int f d\eta\geq  s-\delta\right\},$$
where $H^\ba_\eta(\tau,M;\ell):=\sum_{i=1}^k a_i H_{\eta \circ \tau_{i-1}^{-1}}(\tau, M;\ell)$.
Then by \eqref{e-2014-4}, $\Omega_{M,\ell, \delta}$ is a non-empty compact set whenever $(M,\ell,\delta)\in \E$.
However $$\Omega_{M_1,\ell_1, \delta_1}\cap \Omega_{M_2,\ell_2,\delta_2}\supseteq  \Omega_{M_1+M_2,\ell_1\ell_2, \min\{\delta_1,\delta_2\}}$$
for any $(M_1,\ell_1,\delta_1),(M_2,\ell_2,\delta_2)\in \E$.
It follows (by finite intersection property)  that
$$\bigcap_{(M,\ell,\delta)\in \E} \Omega_{M,\ell, \delta}\neq \emptyset.$$
Take $\mu_s\in \bigcap_{(M,\ell,\delta)\in \E} \Omega_{M,\ell, \delta}$.
Then
$$h^\ba_{\mu_s}(T_1,\tau)+\int f d\mu_s\geq  s,$$
where $h^\ba_{\mu_s}(T_1,\tau):=\sum_{i=1}^k a_i h_{\mu_s\circ \tau_{i-1}^{-1}}(T_i,\tau)$.
Since the map $\theta\in \mathcal{M}(X_1,T_1)\mapsto
h_\theta^{\ba}(T_1,\tau)$ is   upper
semi-continuous (see Lemma \ref{usc}), we can find $\mu\in \M(X_1, T_1)$ such that
$$h^\ba_{\mu}(T_1,\tau)+\int f d\mu \geq P^\ba_W(T_1, f,\epsilon) -\omega_\epsilon(f)$$
by letting $s\nearrow P^\ba_W(T_1, f,\epsilon)$. Since $h^\ba_{\mu}(T_1)\geq h^\ba_{\mu}(T_1,\tau)$, this completes the proof of the proposition.
\end{proof}
%%%%%%%%%%%%%%%%%%%%%%%%%%%%%%%%%%%%%%%%%%%%%%%%%%%%%%%%%%%%%%%%%%%%%%
\section{Sub-additive case}
\label{s-6}
In this section, we extend Theorem \ref{thm-1.1} to sub-additive potentials,  under the following two additional assumptions: (1) $\htop(T_1)<\infty$ and (2) the entropy maps $\theta\in \mathcal{M}(X_i,T_i)\mapsto h_\theta(T_i)$, $i=1,2,\cdots,k$, are upper semi-continuous.

\begin{de}
Let $f:X_1\rightarrow [-\infty,+\infty)$ be an upper semicontinuous
function. Define $\Psi=\{\log \psi_n\}_{n=1}^\infty$ by
$\psi_n(x)=\exp(\sum_{j=0}^{n-1}f(T_1^jx))$. In this case, $\Psi$ is
additive. We just define
 $$P^{\bf a} (T_1,f):=P^{\bf a}
 (T_1,\Psi).$$
\end{de}

\begin{lem} \label{usc-est} Assume that $\htop(T_1)<\infty$ and the entropy maps $\theta\in \mathcal{M}(X_i,T_i)\mapsto h_\theta(T_i)$, $i=1,2,\cdots,k$, are upper semi-continuous. Let  $f:X_1\to
[-\infty,+\infty)$ be a upper semicontinuous function. Then there exists $\mu\in \mathcal{M}(X_1,T_1)$ such
that
$$ h_\mu^{\bf a}(T_1)+\int_{X_1} f d \mu \ge P^{\bf a} (T_1,f). $$
\end{lem}
\begin{proof}  For $g\in C(X_1)$ with $g\ge f$, we define
$$\mathcal{M}_g=\Big\{ \nu\in \mathcal{M}(X_1,T_1):h_\nu^{\bf
a}(T_1)+\int_{X_1} g d\nu \ge P^{\bf a}
(T_1,f)\Big\}.$$
Notice that, under the assumptions of the lemma,
 the entropy map $\nu\in \mathcal{M}(X_1,T_1)\mapsto h^{\ba}_\nu(T_1)$ is a bounded upper semi-continuous function. Hence by  Theorem \ref{thm-1.1},  there exists $\mu_g\in
\mathcal{M}(X_1,T_1)$ such that
$$h_{\mu_g}^{\bf a}(T_1)+\int_{X_1} g d \mu_g\ge  P^{\bf a}
(T_1,g) \ge P^{\bf a} (T_1,f).$$ Thus
$\mu_g\in \mathcal{M}_g$.  Since $\nu\in \mathcal{M}(X_1,T_1)\mapsto
\int_{X_1}g d\nu$ is a  bounded continuous non-negative valued
function on $\mathcal{M}(X_1,T_1)$, the mapping $\nu\in
\mathcal{M}(X_1,T_1)\mapsto h_\nu^{\bf a}(T_1)+\int_{X_1}g
d\nu$ is a bounded upper semicontinuous non-negative valued function
on $\mathcal{M}(X_1,T_1)$. Thus  $\mathcal{M}_g$ is a non-empty
closed subset of $\mathcal{M}(X_1,T_1)$.

Now put
$$\mathcal{M}_f:=\bigcap_{g\in C(X_1),g\ge f} \mathcal{M}_g.$$
Note that $\mathcal{M}_{g_1}\cap \mathcal{M}_{g_2}\supseteq
\mathcal{M}_{\min\{g_1,g_2\}}$ for any $g_1,g_2\in C(X_1)$ with
$g_1\ge f$, $g_2\ge f$, and each $\mathcal{M}_g$ is a non-empty
closed subset of the compact metric space $\mathcal{M}(X_1,T_1)$.
Hence $\mathcal{M}_f\neq \emptyset$, by the finite intersection property characterization of compactness. Take any $\mu\in
\mathcal{M}_f$. Then
$$h_{\mu}^{\bf a}(T_1)+\int_{X_1} g d \mu\ge  P^{\bf a} (T_1,f)$$
for any $g\in C(X_1)$ with $g\ge f$. Moreover, since $0\le
h_{\mu}^{\bf a}(T_1)<\infty$, we have
$$h_{\mu}^{\bf a}(T_1)+\inf \limits_{g\in C(X_1),g\ge f} \int_{X_1} g d \mu\ge  P^{\bf a} (T_1,f).$$
Finally by Lemma \ref{appro}, $\inf \limits_{g\in C(X_1),g\ge f}
\int_{X_1} g d \mu=\int_{X_1}f d \mu$ and thus
$$h_{\mu}^{\bf a}(T_1)+ \int_{X_1} f d \mu\ge  P^{\bf a} (T_1,f).$$
This completes  the proof of the lemma.
\end{proof}

\begin{lem}\label{estimate-1}
Let $\Phi=\{\log \phi_n\}_{n=1}^\infty$ be a sub-additive potential
on $X_1$. If for $\ell\in \mathbb{N}$ and $M\in
\mathbb{N}$, let $f_{\ell,M}(x)=\max\{\frac{1}{\ell}\log \phi_\ell(x),-M\}$ for
$x\in X_1$, then $f_{\ell,M}: X_1\rightarrow \mathbb{R}$ is a bounded
upper semi-continuous function and $$P^{\bf
a}(T_1,f_{\ell,M}) \ge P^{\bf
a}(T_1,\Phi).$$
\end{lem}
\begin{proof} Let $\ell\in \mathbb{N}$ and $M\in \mathbb{N}$.  Let
$f_{\ell,M}=\max\{ \frac{1}{\ell}\log \phi_\ell,-M\}$. It is clear that
$f_{\ell,M}: X_1\rightarrow \mathbb{R}$ is a bounded upper
semi-continuous function since $\frac{1}{\ell}\log
\phi_\ell:X_1\rightarrow [-\infty,+\infty)$ is  upper semi-continuous.

 Let $\phi_0(x)\equiv 1$ for $x\in X_1$ and
$$D:=D(\ell)=\sup \limits_{x\in X_1,\; i\in\{0,1,\cdots,\ell-1\}} \log
\phi_i(x).$$ Then $0\le D<\infty$. For $x\in X_1$ and $n\ge
2\ell$, we have
\begin{align*}
\log \phi_n(x)&\le \log
\phi_i(x)+\left(\sum_{j=0}^{[\frac{n-i}{\ell}]-1}\log
\phi_\ell(T_1^{j\ell+i}x)\right)+\log
\phi_{n-i-[\frac{n-i}{\ell}]\ell}\left(T_1^{i+[\frac{n-i}{\ell}]\ell}x\right)\\
&\le 2D+\sum_{j=0}^{[\frac{n-i}{\ell}]-1}\log \phi_\ell(T_1^{j\ell+i}x)
\end{align*}
for each $i\in\{0,1,\ldots,\ell-1\}$, using the sub-additivity of  $\Phi=\{\log
\phi_n\}_{n=1}^\infty$, where $[a]$
denotes the greatest integer $\leq a$. Summing  $i$ from $0$ to
$\ell-1$, we obtain
\begin{align*}
\log \phi_n(x) &\le
2D+\sum_{i=0}^{\ell-1}\sum_{j=0}^{[\frac{n-i}{\ell}]-1}\frac{1}{\ell}\log
\phi_\ell(T^{j\ell+i}x)=2D+\sum_{j=0}^{n-\ell} \frac{1}{\ell}\log \phi_\ell(T_1^jx)\\
&\le 2D+\sum_{j=0}^{n-\ell} f_{\ell,M}(T_1^jx)\le
C+\sum_{j=0}^{n-1}f_{\ell,M}(T_1^jx)
\end{align*}
where $C=2D+\ell M\in [0,+\infty)$.

Define $\Psi=\{\log \psi_n\}_{n=1}^\infty$ by
$\psi_n(x)=\exp\left(\sum_{j=0}^{n-1}f_{\ell,M}(T_1^jx)\right)$. Then
\begin{equation}\label{gg-eq-1}
\phi_n(x)\le e^C \psi_n(x), \qquad \forall\; x\in X_1,\; n\ge 2\ell,
\end{equation}
This implies that for any $\epsilon>0$, $s\in \mathbb{R}$ and $N\ge 2a_1\ell$,
$$\mathcal{M}^{{\bf a},s}_{\Phi, N,\epsilon}(X_1)\le  e^{\frac{C}{a_1}}\cdot \mathcal{M}^{{\bf a}, s}_{\Psi, N,
\epsilon}(X_1).$$ Hence $\mathcal{M}^{{\bf a},s}_{\Phi,
\epsilon}(X_1)\le e^{\frac{C}{a_1}} \mathcal{M}^{{\bf a}, s}_{\Psi,
\epsilon}(X_1)$ for $\epsilon>0$, $s\in \mathbb{R}$. It follows that  $$P^{\bf
a}(T_1,\Phi,X_1,\epsilon)\le P^{\bf
a}(T_1,\Psi,X_1,\epsilon)=P^{\bf a}(T_1,f_{\ell,M}, X_1,\epsilon) .$$
Letting  $\epsilon \to 0$, we are done.
\end{proof}

\begin{thm} Assume that $\htop(T_1)<\infty$ and the entropy maps $\theta\in \mathcal{M}(X_i,T_i)\mapsto h_\theta(T_i),$
 $i=1,2,\cdots,k$, are upper semi-continuous. Let  $\Phi=\{\log \phi_n\}_{n=1}^\infty$ be a sub-additive potential on
$X_1$. Then
$$ P^{\bf a}(T_1,\Phi)=\sup\{ h_\mu^{\bf a}(T_1)+\Phi_*(\mu): \mu\in \mathcal{M}(X_1,T_1)\},$$
and moreover the supremum is attainable.
\end{thm}
\begin{proof} By Proposition \ref{thm-3.3}, it is sufficient to show that there exists $\mu\in \mathcal{M}(X_1,T_1)$
such that $P^{\bf a}(T_1,\Phi)\le h_\mu^{\bf a}(T_1)+\Phi_*(\mu)$.

For $n,M\in \mathbb{N}$, let $f_n(x)=\frac{1}{n}\log \phi_n(x)$ and $f_{n,M}(x)=\max\{  \frac{1}{n}\log \phi_n(x),-M\}$ for $x\in
X_1$. Then  $f_{n,M}$ is a bounded upper semi-continuous function.
Define $$\mathcal{M}_{n,M}=\left\{ \nu\in \mathcal{M}(X_1,T_1):
h_\nu^{\bf a}(T_1)+\int_{X_1}f_{n,M} \, d \nu\ge P^{\bf
a}(T_1,\Phi)\right\}.$$
By Lemma \ref{usc-est}, there exists
$\mu_{n,M}\in \mathcal{M}(X_1,T_1)$ such that
$$h_{\mu_{n,M}}^{\bf a}(T_1)+\int_{X_1}f_{n,M} \, d \mu_{n,M}\ge P^{\bf
a}(T_1,f_{n,M}) \ge P^{\bf
a}(T_1,\Phi),$$ where the last inequality comes from Lemma
\ref{estimate-1}. Thus $\mu_{n,M}\in \mathcal{M}_{n,M}$. By the assumption, we know that the function $h_{\bullet}^{\bf
a}(T_1)$ is  bounded, upper semi-continuous and  non-negative
 on $\mathcal{M}(X_1,T_1)$. Notice that $\nu\in
\mathcal{M}(X_1,T_1)\mapsto \int_{X_1}f_{n,M} d\nu$ is also  an upper
semi-continuous function from $\mathcal{M}(X_1,T_1)$ to
$\mathbb{R}$. Hence $\nu\in \mathcal{M}(X_1,T_1)\mapsto
h_\nu^{\bf a}(T_1)+\int_{X_1}f_{n,M} d\nu$ is  upper
semi-continuous.    Thus  $\mathcal{M}_{n,M}$ is a non-empty closed
subset of $\mathcal{M}(X_1,T_1)$. Moreover since
$\mathcal{M}_{n,1}\supseteq \mathcal{M}_{n,2}\supseteq \cdots$ and
$\inf_{M\in \mathbb{N}}\int_{X_1}f_{n,M}d\nu=\int_{X_1} f_n d\nu$
for any $\nu\in \mathcal{M}(X_1,T_1)$, one has
$\mathcal{M}_n=\bigcap_{M\in \mathbb{N}}\mathcal{M}_{n,M}$ is a
non-empty closed subset of $\mathcal{M}(X_1,T_1)$.

Now put
$$\mathcal{M}_{\Phi}:=\bigcap_{n\in \mathbb{N}} \mathcal{M}_n.$$
Since $\int_{X_1} f_{n_1n_2} d \nu \le \min\{ \int_{X_1} f_{n_1} d
\nu, \int_{X_1} f_{n_2} d \nu\}$ for $\nu\in \mathcal{M}(X_1,T_1)$,
we have  $\mathcal{M}_{n_1}\cap \mathcal{M}_{n_2}\supseteq
\mathcal{M}_{n_1n_2}$  for any $n_1,n_2\in \mathbb{N}$.  Moreover
since each $\mathcal{M}_n$ is a non-empty closed subset of the
compact metric space $\mathcal{M}(X_1,T_1)$, one has
$\mathcal{M}_{\Phi}\neq \emptyset$ by the finite intersection property characterization of compactness. Take any $\mu\in
\mathcal{M}_{\Phi}$. Then
$$h_{\mu}^{\bf a}(T_1)+\int_{X_1} f_n d \mu\ge  P^{\bf a} (T_1,\Phi)$$
for any $n\in \mathbb{N}$. Moreover, since $0\le h_{\mu}^{\bf
a}(T_1)<\infty$, we have
$$h_{\mu}^{\bf a}(T_1)+\inf \limits_{n\in \mathbb{N}} \frac{1}{n}\int_{X_1} \log \phi_n d \mu\ge  P^{\bf a} (T_1,\Phi).$$
Finally since $\inf \limits_{n\in \mathbb{N}} \frac{1}{n}\int_{X_1}
\log \phi_n d \mu=\Phi_*(\mu)$ and thus
$$h_{\mu}^{\bf a}(T_1)+ \Phi_*(\mu)\ge  P^{\bf a} (T_1,\Phi).$$
This finishes the proof of the Theorem.
\end{proof}

%%%%%%%%%%%%%%%%%%%%%%%%%%%%%%%%%%%%%%%%%%%%%%%%%%%%%%%%

%\noindent {\bf Acknowledgements}  The first author
%was partially supported by the direct grant and RGC grants in the Hong Kong Special Administrative
%Region, China (projects CUHK400706, CUHK401008).  The second author was partially supported by NSFC
%(Grant 11225105, 11371339).
%%%%%%%%%%%%%%%%%%%%%%%%%%%%%%%%%%%%%%%%%%%%%%%%%%%%%%%%%%
\bigskip

\section{Final remarks and examples}
In this section we give some remarks, examples and questions.
\subsection{}
 \label{s-7.1}
 In \cite{BaFe12, Fen11}, Barral and the first author  defined  weighted topological pressure for factor maps between subshifts in a different way, motivated from the study of  multifractal analysis on affine Sierpinski gaskets \cite{BaMe07, BaMe08, Kin95, Ols98}  and a question of Gatzouras and  Peres \cite{GaPe96} on the uniqueness of invariant measures of full dimension on certain affine invariant sets.  The approach is based on the following lemma, which is derived from the relativized variational principle of Ledrappier and Walters \cite{LeWa77} and its sub-additive extension \cite{ZhCa08}.

\begin{lem}
\label{lem-BF}\cite{BaFe12, Fen11}
Assume that  $(X, T)$ and $(Y, S)$ are subshifts over finite alphabets and $\pi:X\to Y$ is a factor map. Let $f\in C(X)$ (or more general, a subadditive potential on $X$). Then there
exists a sub-additive potential $\Phi_f=(\log \phi_n)_{n=1}^\infty$ on $Y$ such that for any $\nu\in \M(Y,S)$,
$$\sup_{\mu\in \M(X,T),\; \mu\circ \pi^{-1}=\nu} \left(\int f d\mu+h_\mu(T)-h_\nu(S)\right)=\Phi_*(\nu):=\lim_{n\rightarrow +\infty}
\frac{1}{n} \int \log \phi_n d \nu.$$
\end{lem}

According to above lemma,  for given $a_1, a_2>0$, one has
\begin{align*}
&\sup_{\mu\in \M(X,T), \; \mu\circ \pi^{-1}=\nu}\left(\int f d\mu+a_1h_\mu(T)+a_2h_\nu(S)\right)\\
&\qquad\mbox{}=\sup_{\nu\in \M(Y,S)} \left\{ (a_1+a_2)h_\nu(S)+\sup_{\mu\in \pi^{-1}\nu} a_1\left(\int \frac{1}{a_1}f d\mu+h_\mu(T)-h_\nu(S)\right)\right\}\\
&\qquad\mbox{}=\sup_{\nu\in \M(Y,S)} \{(a_1+a_2)h_{\nu}(S)+(\Phi_{a_1^{-1}f})_*(\nu)\}\\
&\qquad\mbox{}=(a_1+a_2) P\left(S, \frac{a_1}{a_1+a_2}\Phi_{a_1^{-1}f}\right).
\end{align*}
where the last equality follows from the sub-additive thermodynamic formalism (see e.g. \cite{CFH08}). Hence in \cite{BaFe12, Fen11},  $P^{(a_1,a_2)}(T,f)$ was defined in terms of sub-additive topological pressure in the subshift case.

However, Lemma \ref{lem-BF} does not extend to factor maps between general topological dynamical systems. Below we will give a counter example. Hence the approach in \cite{BaFe12, Fen11} in defining weighted topological pressure does not extend to general topological dynamical systems.

\begin{ex} \label{ex-1} Let $X=\{ (x,y,z)\in \mathbb{R}^3: -1\le x \le 1, y^2+z^2=x^2\}$ be a cone surface. Define $T:X\rightarrow X$ by
$$T((x,x\cos \theta,x\sin \theta))=(x,x\cos(2\theta), x\sin(2\theta)), \quad x\in [-1,1].$$
Let $Y=[-1,1]$ and $S:Y\rightarrow Y$ be the identity.
Set $\pi: X\rightarrow Y$ by $\pi((x,y,z))= x$.
Then $(Y, S)$ is a factor of $(X, T)$ associated with the factor map $\pi$. Take $f\in C(X)$ with  $f\equiv 0$. Suppose that Lemma \ref{lem-BF} extends to this case, that is, there exists a sub-additive potential $\Phi$ on $Y$ such that
for any $\nu\in \M(Y,S)$,
\begin{equation}\label{e-BF}
\sup_{\mu\in \pi^{-1}\nu} (h_\mu(T)-h_\nu(S))=\Phi_*(\nu).
\end{equation}
In what follows we derive a contradiction.

We first claim that the mapping \begin{equation}
\label{e-map}\nu\in \M(Y,S)\mapsto \sup_{\mu\in \pi^{-1}\nu} (h_\mu(T)-h_\nu(S))
\end{equation}
is not upper semi-continuous. To see this, for $t\in Y$, let $\nu_t=\delta_t$ (the Dirac measure at $t$). Clearly $\delta_t\in \M(Y, S)$ and when $t\rightarrow 0$, $\delta_t\rightarrow\delta_0$ in the weak-star topology.
However one can check that
$$\sup_{\mu\in \pi^{-1}\delta_t} \big( h_\mu(T)-h_{\nu_t}(S) \big)=\begin{cases} \log 2, &\text{ for } t\neq 0\\
0, &\text{ if } t=0 \end{cases}.$$
Hence the mapping in \eqref{e-map} is not upper semi-continuous. Therefore by \eqref{e-BF}, $\nu\mapsto \Phi_*(\nu)$ is not upper semi-continuous on $\M(Y,S)$. But this contradicts the fact that $\nu\mapsto \Phi_*(\nu)$ is always upper semi-continuous (see e.g. \cite[Proposition A.1.(2)]{FeHu10}).

\end{ex}

\subsection{}
\label{s-7.2}
Using Corollary \ref{cor-entropy},  we can extend Kenyon-Peres' variational principle \eqref{e-var-H} and its higher dimensional version to a particular class of skew product expanding maps on the $k$-torus $\T^k:=\R^k/\Z^k$ ($k\geq 2$).

 To see this, let $2\leq m_1\leq m_2\leq \ldots\leq m_k$ be integers. For $i=1, \ldots, k-1$, let $\phi_i$ be $C^1$ real-valued functions  on $\T^{i}$. Define $T_1: \T^k\to \T^k$ by
 $$
 T_1((x_1,\ldots,x_k))=(m_1x_1, m_2x_2+\phi_1(x_1), \ldots, m_kx_k+\phi_{k-1}(x_1,\ldots, x_{k-1})).
 $$
  This transformation can be viewed as a skew product of the maps $$x_i\mapsto m_ix_i,\quad  (i=1,\ldots,k).
 $$

Let $K\subset \T^k$ be a $T_1$-invariant compact set.  Let  $\tau_i$ $(i=1,\ldots, k-1$) be the canonical projection from $\T^{k}$ to $\T^{k-i}$, i.e. $$
\tau_i(x_1,\ldots, x_k)=(x_1,\ldots, x_{k-i}).
$$
Set
$X_1= K$ and $X_i=\tau_{i-1}(K)$ for $2\leq i\leq k$.
Define $T_i: X_i\to X_i$ ($i=2,\ldots, k$) by
  $$
 T_i((x_1,\ldots,x_i))=(m_1x_1, m_2x_2+\phi_1(x_1), \ldots, m_ix_i+\phi_{i-1}(x_1,\ldots, x_{i-1})).
 $$
 Then  $(X_{i+1}, T_{i+1})$ is the factor of $(X_i, T_i)$ associated with the factor map  $\pi_i: X_i\to X_{i+1}$, which is defined by $$(x_1,\ldots, x_{k+1-i})\mapsto (x_1,\ldots, x_{k-i}).$$
Define $\ba=(a_1,\ldots, a_k)$ with
$$a_1=\frac{1}{\log m_k},\quad a_i=\frac{1}{\log m_{k+1-i}}-\frac{1}{\log m_{k+2-i}} \quad \mbox{ for }i=2,\ldots, k.$$

It is direct to check that there exist two constants $C_1, C_2>0$ (depending on $\phi_i$'s) such that for any $\epsilon>0$ and $x\in \T^k$,
\begin{equation}
\label{e-equiv}
C_2B_{e^{-n}\epsilon}(x)\subset B_n^{\ba}(x, \epsilon)\subset C_1 B_{e^{-n}\epsilon}(x).
\end{equation}
Hence from the definition of $\htop^{\ba}(\cdot)$, we see that $\htop^{\ba}(T_1, K)=\dim_H K$.  Applying Corollary \ref{cor-entropy}, we have
\begin{equation}
\label{e-var''}
\dim_HK=\htop^{\ba}(T_1, K)=\sup_{\mu\in \M(X_1, T_1)} h^\ba_\mu(T_1),
\end{equation}
where the supremum is attainable at some ergodic $\mu\in \M(X_1, T_1)$. Moreover by \eqref{e-equiv} and Theorem \ref{thm-4.1'}, we have
$\dim_H\mu=h^\ba_\mu(T_1)$ for each ergodic $\mu\in \M(X_1, T_1)$.  Hence there exists an ergodic $\mu\in \M(X_1, T_1)$ of full Hausdorff dimension, i.e. \begin{equation}
\label{e-end}
\dim_H\mu=\dim_HK.
\end{equation}
This extends the work of Kenyon and Peres \cite{KePe96}.  We remark that \eqref{e-end} was also proved by Luzia \cite{Luz06} for a more general class of skew product expanding maps on $\T^2$.

\subsection{}In \cite{FeHu12}, the authors proved a variational principle for topological entropies for arbitrary Borel subsets. We remark that this principle also holds for weighted topological entropies, by applying Lemma \ref{pro-3.1} and following the arguments in \cite{FeHu12}.

 In the end we pose several questions about  possible extensions of Theorem \ref{thm-1.1}:  does this result remain valid for $\Z^d$-actions? and moreover does it admit a relativized or randomized version? is there an analogous topological extension of the dimensional result on Gatzouras-Lalley self-affine carpets \cite{LaGa92}?

\appendix

\section{A weighted version of  the  Brin-Katok theorem}
\label{s-a}

The main result in this appendix is the following weighted version of the Brin-Katok theorem. It is needed in our proof of the lower bound of Theorem.

\begin{thm} \label{thm-4.1'} For each ergodic measure $\mu\in \mathcal{M}(X_1,T_1)$, we have
$$\lim_{\epsilon \rightarrow 0} \liminf_{n\rightarrow +\infty} \frac{-\log \mu(B_n^{\bf a}(x,\epsilon))}{n}=
\lim_{\epsilon \rightarrow 0} \limsup_{n\rightarrow +\infty}
\frac{-\log \mu(B_n^{\bf a}(x,\epsilon))}{n}=h_\mu^{\bf a}(T_1)$$
for $\mu$-a.e. $x\in X_1$.
\end{thm}

When $\ba=(1,0,\ldots, 0)$, the above result reduces to the  Brin-Katok theorem on local entropy \cite{BrKa83}.

The proof of Theorem \ref{thm-4.1'} is based on the following weighted version of the Shannon-McMillan-Breiman theorem.
\begin{pro} \label{SMB} Let $(X,\mathcal{B},\mu,T)$ be a measure preserving
dynamical system and $k\ge 1$. Let $\alpha_1, \ldots,
\alpha_k$ be $k$ countable measurable partitions of
$(X,\mathcal{B},\mu)$ with $H_\mu(\alpha_i)<\infty$ for each $i$, and ${\bf a}=(a_1,\ldots,a_k)\in \mathbb{R}^k$
with $a_1>0$  and $a_i\ge 0$ for $i\ge 2$. Then
\begin{equation}\label{fh-eq-1} \lim_{N\rightarrow +\infty}\frac{1}{N}I_\mu
\Big(\bigvee_{i=1}^k (\alpha_i)_0^{\lceil(a_1+\cdots+a_i)N\rceil-1}
\Big)(x)=\sum_{i=1}^k a_i\mathbb{E}_\mu(F_i|\mathcal{I}_\mu)(x)
\end{equation}
almost everywhere, where $$F_i(x):=I_\mu
\Big(\bigvee_{j=i}^{k}\alpha_j\big|\bigvee_{n=1}^\infty
T^{-n}(\bigvee_{j=i}^{k}\alpha_j)\Big)(x), \quad i=1,\ldots,k$$ and
$\mathcal{I}_\mu=\{ B\in \mathcal{B}:\; \mu(B\triangle T^{-1}B)=0\}$.
In particular,  if $T$ is ergodic, we have
\begin{equation*}
%\label{fh-eq-1-0}
\lim_{N\rightarrow +\infty}\frac{1}{N}I_\mu \Big(\bigvee_{i=1}^k
(\alpha_i)_0^{\lceil(a_1+\cdots+a_i)N\rceil-1} \Big)(x)=\sum_{i=1}^k
a_ih_\mu(T,\bigvee_{j=i}^{k}\alpha_j)
\end{equation*}
almost everywhere.
\end{pro}

When  $k=1$ and $a_1=1$,  Proposition \ref{SMB} reduces to the classical Shannon-McMillan-Breiman theorem (see e.g.  \cite[Theorem 7]{Par81}). We remark that a variant of  Proposition  \ref{SMB}, for certain particular partitions, was proved by Kenyon and Peres (cf. \cite[Lemmas 3.1 and 4.4]{KePe96}) in the  case that $\mu$ is ergodic. For completeness and for the convenience of the reader, we will
provide a full proof of Proposition  \ref{SMB} in the end of this section, by adapting the argument by Kenyon and Peres in \cite{KePe96}.

The following result is a direct corollary of Proposition \ref{SMB}.

\begin{cor}\label{c-SMB} Let $(X,\mathcal{B},\mu,T)$ be an ergodic measure preserving dynamical
system and $k\ge 1$. If $\alpha_1,\ldots, \alpha_k$ are
$k$ countable measurable partitions of $(X,\mathcal{B},\mu)$ with
$\alpha_1\succeq\alpha_2\succeq\cdots\succeq\alpha_k$ and
$H_\mu(\alpha_i)<\infty$, $i=1,\ldots,k$, and ${\bf
a}=(a_1,\ldots,a_k)\in \mathbb{R}^k$ with $a_1>0$  and $a_i\ge
0$ for $i\ge 2$, then
\begin{equation*}
%\label{fh-eq-1}
\lim_{N\rightarrow +\infty}\frac{1}{N}I_\mu \Big(\bigvee_{i=1}^k
\Big(
\bigvee_{j=\lceil(a_0+\cdots+a_{i-1})N\rceil}^{\lceil(a_1+\cdots+a_i)N\rceil-1}T^{-j}\alpha_i\Big)
\Big)(x)=\sum_{i=1}^k a_ih_\mu(T,\alpha_i)
\end{equation*}
almost everywhere, where we make the convention $a_0=0$.
\end{cor}

\begin{proof}[Proof of Theorem \ref{thm-4.1}] We just adapt the proof of Brin and Katok \cite{BrKa83} for their local entropy formula.

We first prove the upper bound.  Let $\epsilon>0$. Let $\alpha_i$ be a finite Borel
partition of $X_i$, $i=1,\ldots,k$, with
$\text{diam}(\alpha_i)<\epsilon$.
%and$\beta_i\succeq \pi_i^{-1}(\beta_{i+1})$ for $i=1,\cdots,k-1$.
Then
$$B_n^{\bf a}(x,\epsilon)\supseteq \bigcap_{i=1}^k (\tau_{i-1}^{-1}\alpha_i)_0^{\lceil (a_1+\cdots+a_i)n\rceil-1}(x)$$
for $x\in X_1$. Hence by Proposition \ref{SMB},  for $\mu$-a.e $x\in X_1$ we have
\begin{align*}
\limsup_{n\rightarrow +\infty} &\frac{-\log \mu(B_n^{\bf
a}(x,\epsilon))}{n}\le \limsup_{n\rightarrow +\infty}
\frac{-\log \mu\Big(\bigcap \limits_{i=1}^k (\tau_{i-1}^{-1}\alpha_i)_0^{\lceil (a_1+\cdots+a_i)n\rceil-1}(x)\Big)}{n}\\
&=\limsup_{n\rightarrow +\infty} \frac{I_\mu\Big(\bigvee \limits_{i=1}^k (\tau_{i-1}^{-1}\alpha_i)_0^{\lceil (a_1+\cdots+a_i)n\rceil-1}\Big)(x)}{n}=\sum_{i=1}^k a_ih_\mu\Big(T_1,\bigvee_{j=i}^k \tau_{j-1}^{-1}\alpha_j\Big)\\
&=\sum_{i=1}^k a_ih_\mu\Big(T_1,\tau_{i-1}^{-1}\Big(\alpha_i\vee \bigvee_{j=i+1}^k \pi_{i}^{-1}\circ \cdots \circ \pi_{j-1}^{-1}\alpha_j\Big)\Big)\\
&=\sum_{i=1}^k a_ih_{\mu\circ \tau_{i-1}^{-1}}\Big(T_i,\alpha_i\vee
\bigvee_{j=i+1}^k
\pi_{i}^{-1}\circ \cdots \circ \pi_{j-1}^{-1}\alpha_j\Big)\\
&\le \sum_{i=1}^k a_i h_{\mu\circ \tau_{i-1}^{-1}}(T_i)=h_\mu^{\bf
a}(T_1).
\end{align*}
Letting $\epsilon\to 0$ in the above inequality, we have
$$\lim_{\epsilon \rightarrow 0} \limsup_{n\rightarrow +\infty}
\frac{-\log \mu(B_n^{\bf a}(x,\epsilon))}{n}\le h_\mu^{\bf a}(T_1).$$
This completes the proof of the upper bound.

Next we prove the lower bound.  It is sufficient to show that for any
$\delta>0$, there exist $\epsilon>0$ and a measurable subset $D$ of
$X_1$ such that $\mu(D)>1-3\delta$ and
$$\liminf_{n\rightarrow +\infty} \frac{-\log
\mu(B_n^{\bf a}(x,\epsilon))}{n}\ge \min\left\{\frac{1}{\delta}, h^{\bf
a}_{\mu}(T_1)-\delta\right\}-2(1+a_1+\cdots+a_k)\delta$$ for any $x\in D$.

Fix $\delta>0$. We are going to find such $\epsilon$ and $D$.
First, we find a finite Borel partition $\alpha_i=\{
A^i_1,A^i_2,\ldots,A^i_{u_i}\}$ of $X_i$,  $i=1,\ldots,k$,   such
that
\begin{itemize}
\item[(1)] $\alpha_i\succeq \pi_i^{-1}(\alpha_{i+1})$ for $i=1,\ldots,k-1$.
\item[(2)] $\sum_{i=1}^k a_ih_{\mu\circ \tau_{i-1}^{-1}}(T_i,\alpha_i)\ge \min\{\frac{1}{\delta}, h_\mu^{\bf a}(T_1)-\delta\}$.
\item[(3)] $\mu\circ \tau_{i-1}^{-1}(\partial \alpha_i)=0$ for $i=1,\ldots,k$.
\end{itemize}

\medskip
Let $M=\max\{u_i:\; 1\leq i\leq k\}$ and $\Lambda=\{1,\ldots,M\}$.
Given $m\in \mathbb{N}$, for ${\bf s}=(s_i)_{i=0}^{m-1},{\bf
t}=(t_i)_{i=0}^{k-1}\in \Lambda^{\{0,1,\cdots,m-1\}}$, the {\it
Hamming distance} between  ${\bf s}$ and ${\bf t}$ is defined to be the
following value
$$\frac{1}{m}\#\left\{ i\in \{0,1,\cdots,m-1\}: s_i\neq t_i\right\}.$$
For ${\bf s}\in \Lambda^{\{0,1,\cdots,m-1\}}$ and $0<\tau\le 1$, let
$Q({\bf s},\tau)$ be the total number of those ${\bf t}\in
\Lambda^{\{0,1,\cdots,m-1\}}$ so that the Hamming distance
between ${\bf s}$ and ${\bf t}$ does not exceed  $\tau$. Clearly,
$$Q_m(\tau):=\max \limits_{{\bf s}\in
\Lambda^{\{0,1,\cdots,m-1\}}}Q({\bf s},\tau)\le \binom{m}{\lceil
m\tau \rceil}M^{\lceil m\tau \rceil}.$$ By the Stirling formula, there exists a small
 $\delta_1>0$ and a positive constant $C:=C(\delta,M)>0$
such that
\begin{equation}\label{esti-1}
\binom{m}{\lceil m\delta_1\rceil}M^{\lceil m\delta_1\rceil}\le
e^{\delta m+C}
\end{equation}
for all $m\in \mathbb{N}$.

\medskip
 For $\eta>0$,  set
\begin{align*}
&U^i_\eta(\alpha_i)=\{x\in X_1:\; B(\tau_{i-1}x,\eta)\not \subseteq
\alpha_i(\tau_{i-1}x)\}, \quad i=1,\ldots,k.
\end{align*}
Then $\bigcap_{\eta>0} U^i_\eta(\alpha_i)=\tau_{i-1}^{-1}(\partial
\alpha_i)$, and hence
$\mu(U^i_\eta(\alpha_i))\rightarrow \mu(\tau_{i-1}^{-1}(\partial
\alpha_i))=0$  as $\eta \to 0$. Therefore,
we can choose $\epsilon>0$ such that $\mu(U^i_\eta(\alpha_i))< \delta_1$
for any $0<\eta\le \epsilon$ and $i=1,\ldots,k$.

By the Birkhoff ergodic theorem, for $\mu$-a.e. $x\in X_1$, we have
\begin{align*}
\lim_{n\rightarrow +\infty} & \frac{1}{\lceil
(a_1+\cdots+a_k)n \rceil}\sum \limits_{i=1}^k \sum_{j=\lceil
(a_0+\cdots+a_{i-1})n\rceil}^{\lceil
(a_1+\cdots+a_i)n\rceil-1}\chi_{U^i_\epsilon(\alpha_i)}(T^j_1x)\\
&=\frac{1}{(a_1+\cdots+a_k)} \sum \limits_{i=1}^k
a_i\mu(U^i_\epsilon(\alpha_i))<\delta_1,
\end{align*}
where we take the convention  $a_0=0$. Thus we can find a large natural number $\ell_0$ such
that $\mu(A_\ell)>1-\delta$ for any $\ell\ge \ell_0$, where
\begin{align*}
A_\ell &=\left\{ x\in X_1:\;\frac{1}{\lceil (a_1+\cdots+a_k)n \rceil}\sum
\limits_{i=1}^k \sum_{j=\lceil (a_0+\cdots+a_{i-1})n\rceil}^{\lceil
(a_1+\cdots+a_i)n\rceil-1}\chi_{U^i_\epsilon(\alpha_i)}(T^j_1x)\le
\delta_1  \text{ for all }n\ge \ell\right\}.
\end{align*}
Since $\tau_{0}^{-1}\alpha_1\succeq
\tau_{1}^{-1}\alpha_2\succeq\cdots\succeq \tau_{k-1}^{-1}\alpha_k$, we have
\begin{align*}
%\label{fh-eq-1-0}
 \lim_{n\rightarrow +\infty}& \frac{-\log \mu
\Big(\bigvee \limits_{i=1}^k \Big(
\bigvee \limits_{j=\lceil(a_0+\cdots+a_{i-1})n\rceil}^{\lceil(a_1+\cdots+a_i)n\rceil-1}T_1^{-j}\tau_{i-1}^{-1}\alpha_i
\Big)(x)\Big)}{n}\\
&=\sum_{i=1}^k a_i h_\mu(T_1,\tau_{i-1}^{-1}\alpha_i)=\sum_{i=1}^k
a_i h_{\mu\circ \tau_{i-1}^{-1}}(T_i,\alpha_i)
%&\ge \min\{\frac{1}{\epsilon},h_\mu^{\bf a}(T_1)-\epsilon\}
\end{align*}
almost everywhere by Corollary \ref{c-SMB}. Hence we can find a
large natural number $\ell_1$ such that $\mu(B_\ell)>1-\delta$ for
any $\ell\ge \ell_1$, where $B_\ell$ is the set of all points $x\in
X_1$ such that
\begin{equation}\label{esti-3}\frac{-\log \mu
\Big(\bigvee_{i=1}^k \Big(
\bigvee_{j=\lceil(a_0+\cdots+a_{i-1})n\rceil}^{\lceil(a_1+\cdots+a_i)n\rceil-1}T_1^{-j}\tau_{i-1}^{-1}\alpha_i
\Big)(x)\Big)}{n}\ge \sum_{i=1}^k a_i h_{\mu\circ
\tau_{i-1}^{-1}}(T_i,\alpha_i)-\delta
\end{equation}
for all
$n\ge \ell$.

Fix $\ell\ge \max\{ \ell_0,\ell_1\}$. Let $E=A_\ell\cap B_\ell$.
Then $\mu(E)>1-2\delta$. For $x\in X_1$ and $n\in
\mathbb{N}$, the unique element $$C(n,x)=(C_j(n,x))_{j=0}^{\lceil
(a_1+\cdots+a_k)n\rceil-1}$$ in $\Lambda^{\{0,1,\cdots,\lceil
(a_1+\cdots+a_k)n\rceil-1\}}$ satisfying that $T_1^jx\in
\tau_{i-1}^{-1}(A^i_{C_j(n,x)})$ for
$\lceil(a_0+\cdots+a_{i-1})n\rceil\le j\le
\lceil(a_1+\cdots+a_i)n\rceil-1$, $i=1,\ldots,k$, is called the
{\it $(\{\alpha_i\}_{i=1}^k,{\bf a};n)$-name of $x$}. Since each point in one atom
$A$ of $\bigvee_{i=1}^k \Big(
\bigvee_{j=\lceil(a_0+\cdots+a_{i-1})n\rceil}^{\lceil(a_1+\cdots+a_i)n\rceil-1}T_1^{-j}\tau_{i-1}^{-1}\alpha_i
\Big)$ has the same $(\{\alpha_i\}_{i=1}^k,{\bf a};n)$-name, we define
$$C(n,A):=C(n,x)$$ for any $x\in A$, which is called the
{\it $(\{\alpha_i\}_{i=1}^k,{\bf a};n)$-name of $A$}.

Now if $y\in B_n^{\bf a}(x,\epsilon)$,  then  for
$i=1,\ldots,k$ and  $\lceil(a_0+\cdots+a_{i-1})n\rceil\le
j\le \lceil(a_1+\cdots+a_i)n\rceil-1$,  either $T_1^jx$ and $T_1^jy$
belong to the same element of $\tau_{i-1}^{-1}\alpha_i$ or
$T_1^jx\in U^i_\epsilon(\alpha_i)$. Hence if $x\in E$, $n\ge \ell$ and
$y\in B_n^{\bf a}(x,\epsilon)$, then the Hamming distance between
$(\{\alpha_i\}_{i=1}^k,{\bf a};n)$-name of $x$ and $y$ does not exceed
$\delta_1$. Furthermore, $B_n^{\bf a}(x,\epsilon)$ is contained in
the set of points $y$ whose $(\{\alpha_i\}_{i=1}^k,{\bf a};n)$-name is
$\delta_1 $-close to $(\{\alpha_i\}_{i=1}^k,{\bf a};n)$-name of $x$. It is
clear that the total number $L_n(x)$ of such $(\{\alpha_i\}_{i=1}^k,{\bf
a};n)$-names admits the following estimate:
\begin{align*}
L_n(x)&\le \binom{\lceil(a_1+\cdots+a_k)n\rceil}{\lceil\lceil(a_1+\cdots+a_k)n\rceil\delta_1\rceil}
M^{\lceil\lceil(a_1+\cdots+a_k)n\rceil\delta_1\rceil}\\
&\le
e^{\delta \lceil(a_1+\cdots+a_k)n\rceil+C}\\
&\le
e^{(a_1+\cdots+a_k)\delta n+C+\delta}
\end{align*}
 where the second
inequality comes from \eqref{esti-1}. More precisely, we have shown
that for any $x\in E$ and $n\ge \ell$,
\begin{equation}\label{esti-2}\begin{aligned}
B^a_n(x,\epsilon)&\subseteq \{ y\in X_1:
C(n,y)\text{ is $\delta_1 $-close to } C(n,x)\}\\
&=\bigcup \Big\{ A\in \bigvee_{i=1}^k \Big(
\bigvee_{j=\lceil(a_0+\cdots+a_{i-1})n\rceil}^{\lceil(a_1+\cdots+a_i)n\rceil-1}T_1^{-j}\tau_{i-1}^{-1}\alpha_i
\Big): C(n,A)\text{ is $\delta_1 $-close to } C(n,x)\Big\}
\end{aligned}
\end{equation}
and
\begin{equation}\label{esti-2-1}\begin{aligned}
&\#\Big\{ A\in \bigvee_{i=1}^k \Big(
\bigvee_{j=\lceil(a_0+\cdots+a_{i-1})n\rceil}^{\lceil(a_1+\cdots+a_i)n\rceil-1}T_1^{-j}\tau_{i-1}^{-1}\alpha_i
\Big): C(n,A)\text{ is $\delta_1 $-close to } C(n,x)\Big\}\\
&\le e^{(a_1+\cdots+a_k)\delta n+C+\delta}.
\end{aligned}
\end{equation}

Now for $n\in \mathbb{N}$, let $E_n$ denote the set of points $x$ in $E$
such that there exists an element $A$ in $\bigvee_{i=1}^k \Big(
\bigvee_{j=\lceil(a_0+\cdots+a_{i-1})n\rceil}^{\lceil(a_1+\cdots+a_i)n\rceil-1}T_1^{-j}\tau_{i-1}^{-1}\alpha_i
\Big)$ with $$\mu(A)>e^{\big(-\sum_{i=1}^k a_i h_{\mu\circ
\tau_{i-1}^{-1}}(T_i,\alpha_i)+(2+a_1+\cdots+a_k)\delta \big)n}$$
and the $(\{\alpha_i\}_{i=1}^k,{\bf a};n)$-name of $A$ is $\delta_1$-close
to the $(\{\alpha_i\}_{i=1}^k,{\bf a};n)$-name of $x$. It is clear that if
$x\in E\setminus E_n$, then for each $A\in \bigvee_{i=1}^k \Big(
\bigvee_{j=\lceil(a_0+\cdots+a_{i-1})n\rceil}^{\lceil(a_1+\cdots+a_i)n\rceil-1}T_1^{-j}\tau_{i-1}^{-1}\alpha_i
\Big)$ whose $(\{\alpha_i\}_{i=1}^k,{\bf a};n)$-name is $\delta_1$-close
to the $(\{\alpha_i\}_{i=1}^k,{\bf a};n)$-name of $x$, one has $$\mu(A)\le
e^{\big(-\sum_{i=1}^k a_i h_{\mu\circ
\tau_{i-1}^{-1}}(T_i,\alpha_i)+(2+a_1+\cdots+a_k)\delta \big)n}.$$
In the following, we wish to estimate the measure of $E_n$ for $n\ge
\ell$.

Let $n\ge \ell$.  Put $$\mathcal{F}_n=\left\{ A\in \bigvee_{i=1}^k
\Big(
\bigvee_{j=\lceil(a_0+\cdots+a_{i-1})n\rceil}^{\lceil(a_1+\cdots+a_i)n\rceil-1}T_1^{-j}\tau_{i-1}^{-1}\alpha_i
\Big): \mu(A)>e^{\big(-\sum \limits_{i=1}^k a_i h_{\mu\circ
\tau_{i-1}^{-1}}(T_i,\alpha_i)+(2+a_1+\cdots+a_k)\delta
\big)n}\right\}.$$ Obviously,
$$\#\mathcal{F}_n\le e^{\big(\sum_{i=1}^k a_i h_{\mu\circ
\tau_{i-1}^{-1}}(T_i,\alpha_i)-(2+a_1+\cdots+a_k)\delta \big)n}$$
since $\mu(X_1)=1$.

Let $x\in E_n$. On the one hand since $x\in B_\ell$, $$\mu
\Big(\bigvee_{i=1}^k \Big(
\bigvee_{j=\lceil(a_0+\cdots+a_{i-1})n\rceil}^{\lceil(a_1+\cdots+a_i)n\rceil-1}T_1^{-j}\tau_{i-1}^{-1}\alpha_i
\Big)(x)\Big)\le e^{\big(-\sum_{i=1}^k a_i h_{\mu\circ
\tau_{i-1}^{-1}}(T_i,\alpha_i)+\delta \big)n}$$ by \eqref{esti-3}.
On the other hand by the definition of $E_n$, there exists $A\in
\mathcal{F}_n$ with the $(\{\alpha_i\}_{i=1}^k,{\bf a};n)$-name of $A$ is
$\delta_1$-close to the $(\{\alpha_i\}_{i=1}^k,{\bf a};n)$-name of $x$,
that is the $(\{\alpha_i\}_{i=1}^k,{\bf a};n)$-name of $A$ is
$\delta_1$-close to the $(\{\alpha_i\}_{i=1}^k,{\bf a};n)$-name of
$$\Big(\bigvee_{i=1}^k
\bigvee_{j=\lceil(a_0+\cdots+a_{i-1})n\rceil}^{\lceil(a_1+\cdots+a_i)n\rceil-1}T_1^{-j}\tau_{i-1}^{-1}\alpha_i
\Big)(x).$$ According to this, we have
\begin{equation}\label{esti-cc}
E_n\subset \bigcup \{B:B\in \mathcal{G}_n\}
\end{equation}
where $\mathcal{G}_n$ denotes the set all elements $B$ in
$\bigvee_{i=1}^k \Big(
\bigvee_{j=\lceil(a_0+\cdots+a_{i-1})n\rceil}^{\lceil(a_1+\cdots+a_i)n\rceil-1}T_1^{-j}\tau_{i-1}^{-1}\alpha_i
\Big)$ satisfying $\mu(B)\le e^{\big(-\sum_{i=1}^k a_i h_{\mu\circ
\tau_{i-1}^{-1}}(T_i,\alpha_i)+\delta \big)n}$ and the
$(\{\alpha_i\}_{i=1}^k,{\bf a};n)$-name of $B$ is $\delta_1$-close to the
$(\{\alpha_i\}_{i=1}^k,{\bf a};n)$-name of $A$ for some $A\in
\mathcal{F}_n$.

Since for each $A\in \mathcal{F}_n$, the total number of  $B$ in
$\bigvee_{i=1}^k \Big(
\bigvee_{j=\lceil(a_0+\cdots+a_{i-1})n\rceil}^{\lceil(a_1+\cdots+a_i)n\rceil-1}T_1^{-j}\tau_{i-1}^{-1}\alpha_i
\Big)$, whose $(\{\alpha_i\},{\bf a};n)$-name is $\delta_1$-close
to the $(\{\alpha_i\},{\bf a};n)$-name of $A$, is upper bounded by
$$ \binom{\lceil(a_1+\cdots+a_k)n\rceil}{\lceil\lceil(a_1+\cdots+a_k)n\rceil\delta_1\rceil}
M^{\lceil\lceil(a_1+\cdots+a_k)n\rceil\delta_1\rceil}\le
e^{(a_1+\cdots+a_k)\delta n+C+\delta}.$$
Hence
$$\#\mathcal{G}_n\le e^{(a_1+\cdots+a_k)\delta n+C+\delta}\cdot
(\# \mathcal{F}_n)\le e^{\left(\sum_{i=1}^k a_i h_{\mu\circ
\tau_{i-1}^{-1}}(T_i,\alpha_i)-2\delta\right)n+C+\delta}.
$$
Moreover
$$\mu(E_n)\le e^{\big(-\sum_{i=1}^k a_i h_{\mu\circ
\tau_{i-1}^{-1}}(T_i,\alpha_i)+\delta \big)n}\cdot
(\#\mathcal{G}_n)\le e^{-\delta n+C+\delta}$$ by \eqref{esti-cc}
and the definition of $\mathcal{G}_n$.

Next we take $\ell_2\ge \ell$ so that $\sum_{n=\ell_2}^\infty
e^{-\delta n+C+\delta}<\delta$. Then $ \mu(\bigcup_{n\ge
\ell_2}E_n)<\delta$. Let $D=E\setminus \bigcup_{n\ge \ell_2}E_n$.
Then $\mu(D)>1-3\delta$. For $x\in D$ and $n\ge \ell_2$, since
$x\in E\setminus E_n$, one has
\begin{align*}
\mu(B_n^{\bf a}(x,\epsilon))&\le e^{(a_1+\cdots+a_k)n+C+\delta}\cdot
e^{\big(-\sum_{i=1}^k a_i h_{\mu\circ
\tau_{i-1}^{-1}}(T_i,\alpha_i)+(2+a_1+\cdots+a_k)\delta \big)n}\\
&=e^{\big(-\sum_{i=1}^k a_i h_{\mu\circ
\tau_{i-1}^{-1}}(T_i,\alpha_i)+2(1+a_1+\cdots+a_k)\delta
\big)n+C+\delta}
\end{align*}
by \eqref{esti-2}, \eqref{esti-2-1}
and the definition of $E_n$. Thus for $x\in D$,
\begin{align*}
\liminf_{n\rightarrow +\infty} \frac{-\log \mu(B_n^{\bf
a}(x,\epsilon))}{n}&\ge \sum_{i=1}^k a_i h_{\mu\circ
\tau_{i-1}^{-1}}(T_i,\alpha_i)-2(1+a_1+\cdots+a_k)\delta\\
&\ge\min\left\{\frac{1}{\delta}, h^{\bf
a}_{\mu}(T_1)-\delta\right\}-2(1+a_1+\cdots+a_k)\delta.
\end{align*}
 This finishes
the proof of Theorem \ref{thm-4.1'}.
\end{proof}

In the remaining part of this section, we provide a full proof of Proposition \ref{SMB}. First we give two lemmas.

\begin{lem}[cf. \cite{Par81}]
\label{Martingale}
Let $(X,\mathcal{B},\mu,T)$ be a measure preserving
dynamical system. Let $\alpha,\beta$ be two countable measurable
partitions of $(X,\mathcal{B},\mu)$ with $H_\mu(\alpha)<\infty,
H_\mu(\beta)<\infty$ and $\mathcal{A}$ a sub-$\sigma$-algebra of
$\mathcal{B}$. Let $I_\mu(\cdot|\cdot)$ denote the conditional information of $\mu$. Then we have the following:
\begin{itemize}
\item[(i)] $I_\mu(\alpha|\mathcal{A})\circ
T=I_\mu(T^{-1}\alpha|T^{-1}\mathcal{A})$.

\item[(ii)] $I_\mu(\alpha\vee \beta|\mathcal{A})=I_\mu(\alpha|\mathcal{A})+I_\mu(\beta|\alpha\vee
\mathcal{A})$. In particular, $H_\mu(\alpha\vee
\beta|\mathcal{A})=H_\mu(\alpha|\mathcal{A})+H_\mu(\beta|\alpha\vee
\mathcal{A})$.

\item[(iii)] If $\mathcal{A}_1\subset \mathcal{A}_2\subset \cdots$
is an increasing sub-$\sigma$-algebra of $\mathcal{B}$ with
$\mathcal{A}_n\uparrow \mathcal{A}$, then
$I_\mu(\alpha|\mathcal{A}_n)$ converges almost everywhere and in
$L^1$ to $I_\mu(\alpha|\mathcal{A})$. In particular,
$\lim_{n\rightarrow +\infty}
H_\mu(\alpha|\mathcal{A}_n)=H_\mu(\alpha|\mathcal{A})$.
\end{itemize}
\end{lem}

\begin{lem}\label{Breiman} Let $(X,\mathcal{B},\mu,T)$ be a measure preserving
dynamical system and $F_n\in L^1(X,\mathcal{B},\mu)$ be a sequence
that converges almost everywhere and in $L^1$ to $F\in
L^1(X,\mathcal{B},\mu)$ and $\int_X \sup_k |F_n(x)|
d\mu(x)<+\infty$. If $f:\mathbb{N}\rightarrow \mathbb{N}$ satisfies
$f(n)\ge n$ for all $k\in \mathbb{N}$, then
$$\lim \limits_{n\rightarrow
+\infty}\frac{1}{n}\sum_{j=0}^{n-1}F_{f(n)-j}(T^jx)=\mathbb{E}_\mu(F|\mathcal{I}_\mu)(x)$$
almost everywhere and in $L^1$, where $\mathcal{I}_\mu=\{ B\in
\mathcal{B}: \mu(B\Delta T^{-1}B)=0\}$ and $\mathbb{E}_\mu(F|\mathcal{I}_\mu)$ stands for the conditional expectation of $F$ given $\mathcal{I}_\mu$.
\end{lem}
\begin{proof} This is a slight variant of Maker's ergodic theorem \cite{Mak40}. For the convenience of the reader, we give a detailed proof.  Since $F\in L^1(X,\mathcal{B},\mu)$, by Birkhoff's ergodic
theorem, we have
$$ \lim \limits_{n\rightarrow
+\infty}\frac{1}{n}\sum_{j=0}^{n-1}F(T^jx)=\mathbb{E}_\mu(F|\mathcal{I}_\mu)(x)$$
almost everywhere and in $L^1$. Since
$$\frac{1}{n}\sum_{j=0}^{n-1}F_{f(n)-j}(T^jx)=\frac{1}{n}\sum_{j=0}^{n-1} F(T^jx)+\frac{1}{n}\sum_{j=0}^{n-1}(F_{f(n)-j}(T^jx)-F(T^jx)),$$
 it is suffices to show that
$$\lim \limits_{n\rightarrow
+\infty}\frac{1}{n}\sum_{j=0}^{n-1}|F_{f(n)-j}(T^jx)-F(T^jx)|=0
$$
almost everywhere and in $L^1$. Set $Z_m(x)=\sup_{j\ge
m}|F_j(x)-F(x)|$ for $m\in \N$. Then $0\le Z_m(x)\le \sup_n |F_n(x)|+|F(x)|$ and
$Z_m(x)\to 0$ as $m\rightarrow +\infty$ almost everywhere. Since $\sup_n|F_n(x)|+|F(x)|\in L^1(X,\mathcal{B},\mu)$,
we have $\lim_{m\rightarrow +\infty} \int  Z_m(x) d\mu(x)=0$  by Lebesgue's
dominated convergence theorem. Then we have
$\mathbb{E}_\mu(Z_m|\mathcal{I}_\mu)\rightarrow 0$ as $m\rightarrow
+\infty$ almost everywhere and in $L^1$ (cf. \cite[Theorem 34.2]{Bil95}).

Now let $m\in \mathbb{N}$. For $n>m+1$,
\begin{align*}
 \frac{1}{n} & \sum_{j=0}^{n-1} |F_{f(n)-j}(T^jx)-F(T^jx)|\\
&\le
\frac{1}{n}\sum_{j=n-m}^{n-1}|F_{f(n)-j}(T^jx)-F(T^jx)|+\frac{1}{n}\sum_{j=0}^{n-m-1}
Z_m(T^jx)\\
&\le
\frac{1}{n}\sum_{j=n-m}^{n-1}Z_1(T^jx)+\frac{n-m}{n}\Big(\frac{1}{n-m}\sum_{j=0}^{n-m-1}
Z_m(T^jx)\Big).
\end{align*}

Letting $n\rightarrow +\infty$ and using Birkhoff's ergodic theorem we
have $$\limsup_{n\rightarrow +\infty}\frac{1}{n}\sum_{j=0}^{n-1}
|F_{f(n)-j}(T^jx)-F(T^jx)|\le
\mathbb{E}_\mu(Z_m|\mathcal{I}_\mu)(x)$$ almost everywhere.
 Since
$\mathbb{E}_\mu(Z_m|\mathcal{I}_\mu)\to 0 $ almost everywhere and in $L^1$ as $m\to \infty$,  we have
$$\limsup_{n\rightarrow +\infty}\frac{1}{n}\sum_{j=0}^{n-1} |F_{f(n)-j}(T^jx)-F(T^jx)|=0$$
almost everywhere and in $L^1$, as desired.
\end{proof}

\begin{proof}[Proof of Proposition \ref{SMB}] Our proof is adapted from the arguments of Kenyon and Peres in \cite[Lemmas 3.2, 4.4]{KePe96}.

First we  show that for any $a>0$, $b\ge 0$ and a countable measurable partition $\beta$ of
$(X,\mathcal{B},\mu)$ with $H_\mu(\beta)<\infty$,
\begin{equation}\label{fh-eq-2}
\lim_{N\rightarrow +\infty}\frac{1}{N}I_\mu \Big(
\beta_{\lceil aN \rceil}^{\lceil (a+b)N \rceil-1}
\Big)(x)=b\mathbb{E}_\mu(G|\mathcal{I}_\mu)(x)
\end{equation}
almost everywhere, where $G(x):=I_\mu
\Big(\beta|\bigvee_{n=1}^\infty T^{-n}\beta \Big)(x)$.

If $b=0$, then $\beta_{\lceil aN\rceil}^{\lceil (a+b)N \rceil-1}=\{X, \emptyset\}\ (\text{mod}\ \mu)$
for each $N\in\mathbb{N}$ and so \eqref{fh-eq-2} holds. Now assume
that $b>0$. Note that
$$I_\mu \Big( \bigvee_{n=\lceil aN \rceil}^{\lceil(a+b)N\rceil-1}T^{-n}\beta \Big)(x)=I_\mu
\Big(\bigvee_{n=0}^{\lceil(a+b)N\rceil-1}T^{-n}\beta \Big)(x)-I_\mu(
\bigvee_{n=0}^{\lceil aN\rceil-1}T^{-n}\beta|
\bigvee_{n=\lceil aN \rceil}^{\lceil (a+b)N\rceil-1}T^{-n}\beta\Big)(x).$$ By
the Shannon-McMillan-Breiman theorem, \eqref{fh-eq-2} is equivalent to
\begin{equation}\label{fh-eq-3}
\lim_{N\rightarrow +\infty}\frac{1}{N}I_\mu( \bigvee_{n=0}^{\lceil
aN\rceil-1}T^{-n}\beta| \bigvee_{n=\lceil aN\rceil}^{\lceil
(a+b)N\rceil-1}T^{-n}\beta\Big)(x)=a\mathbb{E}_\mu(G|\mathcal{I}_\mu)(x)
\end{equation}
almost everywhere.

Note that
\begin{align*}
 I_\mu & \Big( \bigvee_{n=0}^{\lceil aN\rceil-1}T^{-n}\beta|
\bigvee_{n=\lceil aN\rceil}^{\lceil(a+b)N\rceil-1}T^{-n}\beta\Big)(x)\\
&=I_\mu\Big(\beta|\bigvee_{n=1}^{\lceil
(a+b)N\rceil-1}T^{-n}\beta\Big)(x)+I_\mu\Big(\bigvee_{n=1}^{\lceil aN
\rceil-1}T^{-n}\beta|
\bigvee_{n=\lceil aN \rceil}^{\lceil(a+b)N\rceil-1}T^{-n}\beta\Big)(x)\\
&=I_\mu\Big(\beta|\bigvee_{n=1}^{\lceil(a+b)N
\rceil-1}T^{-n}\beta\Big)(x)+I_\mu\Big(\bigvee_{n=0}^{\lceil aN
\rceil-2}T^{-n}\beta|
\bigvee_{n=\lceil aN \rceil-1}^{\lceil (a+b)N \rceil-2}T^{-n}\beta\Big)(Tx)\\
&\qquad\qquad \vdots  \\
&=\sum_{j=0}^{\lceil aN
\rceil-1}I_\mu\Big(\beta|\bigvee_{n=1}^{[(a+b)N]-1-j}T^{-n}\beta\Big)(T^jx).
\end{align*}
Write $G_k(x)=I_\mu(\beta|\bigvee_{n=1}^{k-1}T^{-n}\beta)(x)$
for $k\in \mathbb{N}$ and $x\in X$. Then
\begin{equation}\label{fh-eq-4}
I_\mu( \bigvee_{n=0}^{\lceil aN \rceil-1}T^{-n}\beta|
\bigvee_{n=\lceil aN \rceil}^{\lceil (a+b)N
\rceil-1}T^{-n}\beta\Big)(x)=\sum_{j=0}^{\lceil
aN\rceil-1}G_{\lceil(a+b)N\rceil-j}(T^jx).
\end{equation}
Since $\bigvee_{n=1}^{k-1} T^{-n}\beta\uparrow \bigvee_{n=1}^\infty
T^{-n}\beta$ when $k\rightarrow +\infty$,   $G_k\in
L^1(X,\mathcal{B},\mu)$ is a sequence that converges almost
everywhere and in $L^1$ to $G\in L^1(X,\mathcal{B},\mu)$ by Lemma \ref{Martingale}. As
$H_\mu(\beta)<\infty$,  we have $\int_X \sup_k |G_k(x)| d\mu(x)\le
H_\mu(\beta)+1<\infty$ by Chung's lemma \cite{Chung}. By \eqref{fh-eq-4}
and Lemma \ref{Breiman},
\begin{align*}
\lim_{N\rightarrow +\infty} & \frac{1}{N}I_\mu\Big( \bigvee_{n=0}^{\lceil
aN \rceil-1}T^{-n}\beta \Big| \bigvee_{n=\lceil aN \rceil}^{\lceil (a+b)N
\rceil-1}T^{-n}\beta\Big)(x)\\
&=a\lim_{N\rightarrow
+\infty}\frac{1}{\lceil aN\rceil }\sum_{j=0}^{\lceil
aN\rceil-1}G_{\lceil(a+b)N\rceil-j}(T^jx)\\
&=a\mathbb{E}_\mu(G|\mathcal{I}_\mu)(x)
\end{align*}
almost everywhere. Hence \eqref{fh-eq-3} holds, so does \eqref{fh-eq-2}.

\medskip
Now we are ready to prove \eqref{fh-eq-1}, by induction on $k$.
For $k=1$, \eqref{fh-eq-1} reduces to  the
Shannon-McMillan-Breiman theorem. Assume that \eqref{fh-eq-1} holds for $k=\ell$
($\ell\ge 1$). We show below that it holds for $k=\ell+1$.

Let $k=\ell+1$. Write $\beta_i=\bigvee_{j=i}^{\ell+1}\alpha_i$ for
$i=1,\ldots,\ell+1$. Then $\beta_1\succeq
\beta_2\succeq\cdots\succeq \beta_{\ell+1}$ and
$F_i(x)=I_\mu(\beta_i|\bigvee_{n=1}^{+\infty}T^{-n}\beta_i)(x)$ for
$i=1,\ldots,\ell+1$. Note that
\begin{equation}\label{fh-eq-eq}
\bigvee_{i=1}^{\ell+1}
(\alpha_i)_0^{\lceil(a_1+\cdots+a_i)N\rceil-1}=\Big(
\bigvee_{i=1}^\ell (\beta_i)_0^{\lceil(a_1+\cdots+a_i)N\rceil-1}
\Big) \vee (\beta_{\ell+1})_{\lceil
(a_1+\cdots+a_\ell)N\rceil}^{\lceil
(a_1+\cdots+a_\ell+a_{\ell+1})N\rceil-1}.
\end{equation}

By the induction assumption and \eqref{fh-eq-2}, we have
\begin{equation}
\label{fh-eq-5}
\begin{aligned} &\lim_{N\rightarrow
+\infty}\frac{1}{N}I_\mu \Big( \bigvee_{i=1}^\ell (\beta_i)_0^{\lceil(a_1+\cdots+a_i)N\rceil-1}\Big)(x)=\sum_{i=1}^\ell a_i\mathbb{E}_\mu(F_i|\mathcal{I}_\mu)(x) \text{ and }\\
&\lim_{N\rightarrow +\infty}\frac{1}{N}I_\mu\Big(
(\beta_{\ell+1})_{\lceil (a_1+\cdots+a_\ell)N\rceil}^{\lceil
(a_1+\cdots+a_\ell+a_{\ell+1})N\rceil-1})\Big)(x)=a_{\ell+1}
\mathbb{E}_\mu(F_{\ell+1}|\mathcal{I}_\mu)(x)
\end{aligned}
\end{equation}
almost everywhere. Next we use the idea employed by Algoet
and Cover \cite{AC} in their elegant ``sandwich'' proof of the
Shannon-McMillan-Breiman theorem. For  $\mu$-a.e. $x\in X$, we
define
$$Z_m(x)=\frac{\mu \Big( \bigvee \limits_{i=1}^\ell (\beta_i)_0^{\lceil(a_1+\cdots+a_i)m\rceil-1}(x)\Big)\cdot
\mu\Big(  (\beta_{\ell+1})_{\lceil
(a_1+\cdots+a_\ell)m\rceil}^{\lceil
(a_1+\cdots+a_\ell+a_{\ell+1})m\rceil-1}(x)  \Big)} {\mu\Big(\Big(
\bigvee \limits_{i=1}^\ell
(\beta_i)_0^{\lceil(a_1+\cdots+a_i)m\rceil-1}  \vee
(\beta_{\ell+1})_{\lceil (a_1+\cdots+a_\ell)m\rceil}^{\lceil
(a_1+\cdots+a_\ell+a_{\ell+1})m\rceil-1}\Big)(x)\Big)}$$ for all
$m\in \mathbb{N}$. Then  for $\mu$-a.e. $x\in X$, $Z_m(x)>0$ for all
$m\in \mathbb{N}$.

Since \begin{align*}\int_X Z_m(x)d \mu(x)&=\sum_{A\in
 \bigvee \limits_{i=1}^\ell (\beta_i)_0^{\lceil(a_1+\cdots+a_i)m\rceil-1} \atop{B\in
(\beta_{\ell+1})_{\lceil (a_1+\cdots+a_\ell)m\rceil}^{\lceil
(a_1+\cdots+a_\ell+a_{\ell+1})m\rceil-1}}} \int_{A\cap
B}\frac{\mu(A)\mu(B)}{\mu(A\cap B)}d\mu(x)\\
&=\sum_{A\in
 \bigvee \limits_{i=1}^\ell (\beta_i)_0^{\lceil(a_1+\cdots+a_i)m\rceil-1} \atop{B\in
(\beta_{\ell+1})_{\lceil (a_1+\cdots+a_\ell)m\rceil}^{\lceil
(a_1+\cdots+a_\ell+a_{\ell+1})m\rceil-1}}} \mu(A)\mu(B)\\
&=1,
\end{align*}
the series $\sum_{m=1}^\infty \mu(\{x\in X: Z_m(x)\ge e^{\epsilon
m}\})$ converges for every $\epsilon>0$ and the Borel-Canteli Lemma
implies that $\limsup_{N\rightarrow +\infty} \frac{1}{N} \log
Z_N(x)\le 0$ for $\mu$-a.e. $x\in X$. Using the definition of $Z_m$,
\eqref{fh-eq-eq} and \eqref{fh-eq-5},  we obtain
$$\limsup_{N\rightarrow +\infty}\frac{1}{N}I_\mu \Big(\bigvee_{i=1}^{\ell+1}
(\alpha_i)_0^{\lceil(a_1+\cdots+a_i)N\rceil-1}\Big)(x)\le
\sum_{i=1}^{\ell+1}a_i \mathbb{E}_\mu(F_i|\mathcal{I}_\mu)(x)$$ for
$\mu$-a.e. $x\in X$.

Conversely, by \eqref{fh-eq-2} and the induction assumption, we have
\begin{equation}
\label{fh-eq-6}
\begin{aligned} &\lim_{N\rightarrow +\infty}\frac{1}{N}I_\mu\Big(
(\beta_{i})_{\lceil (a_1+\cdots+a_\ell)N\rceil}^{\lceil
(a_1+\cdots+a_\ell+a_{\ell+1})N\rceil-1}\Big)(x)=a_{\ell+1}
\mathbb{E}_\mu(F_{i}|\mathcal{I}_\mu)(x), \ i=\ell,\; \ell+1
\text{ and }\\
& \lim_{N\rightarrow +\infty}  \frac{1}{N}I_\mu \Big(
\bigvee_{i=1}^{\ell-1}
(\beta_i)_0^{\lceil(a_1+\cdots+a_i)N\rceil-1}\vee
(\beta_\ell)_0^{\lceil (a_1+\cdots+a_\ell+a_{\ell+1})N\rceil -1}
\Big)(x)\\
&\quad =(a_\ell+a_{\ell+1})\mathbb{E}_\mu(F_\ell|\mathcal{I}_\mu)(x)+\sum_{i=1}^{\ell-1}
a_i\mathbb{E}_\mu(F_i|\mathcal{I}_\mu)(x)
\end{aligned}
\end{equation}
almost everywhere. Then for  $\mu$-a.e. $x\in X$, we define
\begin{align*}
R_m(x)=&\frac{\mu\Big(\Big( \bigvee \limits_{i=1}^\ell
(\beta_i)_0^{\lceil(a_1+\cdots+a_i)m\rceil-1}  \vee
(\beta_{\ell+1})_{\lceil (a_1+\cdots+a_\ell)m\rceil}^{\lceil
(a_1+\cdots+a_\ell+a_{\ell+1})m\rceil-1}\Big)(x)\Big)} {
\mu\Big(\Big( \bigvee \limits_{i=1}^{\ell-1}
(\beta_i)_0^{\lceil(a_1+\cdots+a_i)N\rceil-1}\vee
(\beta_\ell)_0^{\lceil (a_1+\cdots+a_\ell+a_{\ell+1})N\rceil -1}
\Big)(x)\Big)}\\
&\times \frac{\mu\Big( (\beta_{\ell})_{\lceil
(a_1+\cdots+a_\ell)m\rceil}^{\lceil
(a_1+\cdots+a_\ell+a_{\ell+1})m\rceil-1}(x)  \Big)}{\mu\Big(
(\beta_{\ell+1})_{\lceil (a_1+\cdots+a_\ell)m\rceil}^{\lceil
(a_1+\cdots+a_\ell+a_{\ell+1})m\rceil-1}(x)\Big) }
\end{align*}
 for all
$m\in \mathbb{N}$. Then  for $\mu$-a.e. $x\in X$, $R_m(x)>0$ for all
$m\in \mathbb{N}$.

Since $\beta_\ell\succeq \beta_{\ell+1}$, we have
\begin{align*}\int_X R_m(x)d \mu(x)&=\sum_{A\in
\bigvee \limits_{i=1}^\ell
(\beta_i)_0^{\lceil(a_1+\cdots+a_i)m\rceil-1}
 \atop{ B\in (\beta_{\ell+1})_{\lceil (a_1+\cdots+a_\ell)m\rceil}^{\lceil
(a_1+\cdots+a_\ell+a_{\ell+1})m\rceil-1} \atop{ C\in
(\beta_{\ell})_{\lceil (a_1+\cdots+a_\ell)m\rceil}^{\lceil
(a_1+\cdots+a_\ell+a_{\ell+1})m\rceil-1}}}} \int_{A\cap
B\cap C}\frac{\mu(A\cap B)\mu(B\cap C)}{\mu(A\cap B\cap C)\mu(B)}d\mu(x)\\
&=\sum_{A\in \bigvee \limits_{i=1}^\ell
(\beta_i)_0^{\lceil(a_1+\cdots+a_i)m\rceil-1}
 \atop{ B\in (\beta_{\ell+1})_{\lceil (a_1+\cdots+a_\ell)m\rceil}^{\lceil
(a_1+\cdots+a_\ell+a_{\ell+1})m\rceil-1} \atop{ C\in
(\beta_{\ell})_{\lceil (a_1+\cdots+a_\ell)m\rceil}^{\lceil
(a_1+\cdots+a_\ell+a_{\ell+1})m\rceil-1}}}}  \frac{\mu(A\cap B)\mu(B\cap C)}{\mu(B)}\\
%&=\sum_{B\in (\beta_{\ell+1})_{\lceil
%(a_1+\cdots+a_\ell)m\rceil}^{\lceil
%(a_1+\cdots+a_\ell+a_{\ell+1})m\rceil-1} \atop{ C\in
%(\beta_{\ell})_{\lceil (a_1+\cdots+a_\ell)m\rceil}^{\lceil
%(a_1+\cdots+a_\ell+a_{\ell+1})m\rceil-1}}} \Big(\sum_{A\in \bigvee
%\limits_{i=1}^\ell
%(\beta_i)_0^{\lceil(a_1+\cdots+a_i)m\rceil-1}} \frac{\mu(A\cap B)\mu(B\cap C)}{\mu(B)}\Big)\\
%&=\sum_{B\in (\beta_{\ell+1})_{\lceil
%(a_1+\cdots+a_\ell)m\rceil}^{\lceil
%(a_1+\cdots+a_\ell+a_{\ell+1})m\rceil-1} \atop{ C\in
%(\beta_{\ell})_{\lceil (a_1+\cdots+a_\ell)m\rceil}^{\lceil
%(a_1+\cdots+a_\ell+a_{\ell+1})m\rceil-1}}} \mu(B\cap C)\\
&=1
\end{align*}
for $m\in \mathbb{N}$. Thus the series $\sum_{m=1}^\infty \mu(\{x\in
X: R_m(x)\ge e^{\epsilon m}\})$ converges for every $\epsilon>0$ and
the Borel-Canteli Lemma implies that $\limsup_{N\rightarrow +\infty}
\frac{1}{N} \log R_N(x)\le 0$ for $\mu$-a.e. $x\in X$. Using the
definition $R_N$, \eqref{fh-eq-eq}  and \eqref{fh-eq-6}, we have
$$\liminf_{N\rightarrow +\infty}\frac{1}{N}I_\mu
\Big(\bigvee_{i=1}^{\ell+1}
(\alpha_i)_0^{\lceil(a_1+\cdots+a_i)N\rceil-1}\Big)(x)\ge
\sum_{i=1}^{\ell+1}a_i \mathbb{E}_\mu(F_i|\mathcal{I}_\mu)(x)$$ for
$\mu$-a.e. $x\in X$. for $\mu$-a.e. $x\in X$. This
completes the proof of Proposition \ref{SMB}.
\end{proof}

\noindent{\bf Acknowledgements}.   The first author was partially supported by RGC grants in the Hong Kong Special Administrative
Region, China (projects CUHK401112, CUHK401013). The second author was partially supported by NNSF (11225105, 11371339, 11431012).


\begin{thebibliography}{30}
\bibitem{AC} P.~H. Algoet and T.~M. Cover, A sandwich proof of the Shannon-McMillan-Breiman theorem. {\it Ann. Probab.} {\bf 16}  (1988),   899--909.

%\bibitem{M} R. Mane,  Ergodic theory and differentiable dynamics.
%Ergebnisse der Mathematik und ihrer Grenzgebiete (3), 8. Springer,
%Berlin, 1987.


\bibitem{BaFe12} J. Barral and D. J. Feng,  Weighted thermodynamic formalism on subshifts and applications.   {\it Asian J. Math.} {\bf  16} (2012), 319--352.

\bibitem{BaMe07} J. Barral and M.  Mensi,  Gibbs measures on self-affine Sierpinski carpets and their singularity spectrum.
{\it Ergodic Theory Dynam. Systems} {\bf 27} (2007),  1419--1443.

\bibitem{BaMe08}
J. Barral and M.  Mensi,  Multifractal analysis of Birkhoff averages on `self-affine' symbolic spaces.
{\it Nonlinearity} {\bf 21} (2008), 2409--2425.

\bibitem{Bed84}T. Bedford,
 Crinkly curves, Markov partitions and box dimension in self-similar sets. Ph.D. Thesis, University of Warwick, 1984.


\bibitem{Bil95} P.  Billingsley,
{\it Probability and measure}.
Third edition. Wiley Series in Probability and Mathematical Statistics. A Wiley-Interscience Publication. John Wiley \& Sons, Inc., New York, 1995.
\bibitem{BrKa83} M. Brin and A. Katok, On local entropy.
Geometric dynamics (Rio de Janeiro, 1981), 30--38, {\it Lecture Notes in
Math.}, {\bf 1007}, Springer, Berlin, 1983.


%\bibitem{Bow72} R.  Bowen,  Entropy-expansive maps. {\it Trans. Amer. Math. Soc.} {\bf 164} (1972),  323--331.

\bibitem{Bow73}R. Bowen, Topological entropy for noncompact sets.
{\it Trans. Amer. Math. Soc.} {\bf 184} (1973), 125--136.

\bibitem{Bow79}
R. Bowen, Hausdorff dimension of quasicircles.
{\em Inst. Hautes \'Etudes Sci. Publ. Math.} {\bf 50} (1979), 11--25.

%\bibitem{Bar96}
%L.  Barreira,  A non-additive thermodynamic formalism and
%applications to dimension theory of hyperbolic dynamical systems.
%{\it Ergodic Theory Dynam. Systems} {\bf 16} (1996), no. 5,
%871--927.

%\bibitem{Bar06}
%L. Barreira,
% Nonadditive thermodynamic formalism: equilibrium and Gibbs measures.
%{\it Discrete Contin. Dyn. Syst.}, {\bf 16} (2006),  279--305.
%
%\bibitem{Bar-book} L. Barreira, {\it Dimension and recurrence in hyperbolic dynamics}. Progress in Mathematics, {\bf 272}.
%Birkh?ser Verlag, Basel, 2008.

\bibitem{CFH08}
Y. L. Cao, D. J. Feng and W. Huang, The thermodynamic formalism for sub-additive potentials. {\it Discrete Contin. Dyn. Syst.} {\bf 20} (2008),  639--657.

\bibitem{Chung} K. L. Chung, A note on the ergodic theorem of information theory. {\it Ann. Math. Stat.}
{\bf 32} (1961), 612-614.


%\bibitem{CM}R. Cawley and R. D.  Mauldin, Multifractal decompositions of Moran fractals.
%{\it Adv.  Math.} 92 (1992), no. 2, 196--236.


%\bibitem{Con-book}J. B. Conway, A course in functional analysis.
%Springer-Verlag, New York, 1985.


\bibitem{DS} T. Downarowicz and J. Serafin, Fiber entropy and conditional variational principles in compact non-metrizable spaces.
 {\it Fund. Math.} 172 (2002), no. 3, 217?-247.
%\bibitem{DoMa} T. Downarowicz, A. Maass, Smooth interval maps have symbolic extensions. To appear
%in {\it Inventiones Mathematicae}.


\bibitem{Fal03} K.J. Falconer. {\it Fractal geometry. Mathematical foundations and applications.} Second edition. John Wiley \& Sons, Inc., Hoboken, NJ, 2003.


%\bibitem{Fal88} K. J. Falconer, A subadditive thermodynamic formalism for
%mixing repellers. {\it J. Phys. A} {\bf 21} (1988), no. 14,
%L737--L742.

%\bibitem{FaSl09} K. J. Falconer and  A. Sloan,
%Continuity of subadditive pressure for self-affine sets. To appear in {\it Real Analysis Exchange}.


\bibitem{Fed69} H. Federer, {\it  Geometric measure theory}.  Springer-Verlag, New York Inc.,  1969.
\bibitem{Fen11} D. J. Feng,  Equilibrium states for factor maps between subshifts.  {\it  Adv. Math.} {\bf 226} (2011),  2470--2502.

\bibitem{FeHu10} D. J. Feng and W. Huang, Lyapunov spectrum of asymptotically sub-additive potentials. {\it Comm. Math. Phys.} {\bf 297} (2010), 1--43.

\bibitem{FeHu12} D. J. Feng and W. Huang,
Variational principles for topological entropies of subsets.   {\it J. Funct. Anal.}  {\bf 263} (2012),  2228--2254.

\bibitem{GaPe96}
D. Gatzouras and Y.  Peres,  The variational principle for Hausdorff dimension: a survey. Ergodic theory of $Z\sp d$ actions (Warwick, 1993--1994), 113--125, London Math. Soc. Lecture Note Ser., 228, Cambridge Univ. Press, Cambridge, 1996.


\bibitem{How95} J. D.  Howroyd,  On dimension and on the existence of sets of finite positive Hausdorff measure.
{\it Proc. London Math. Soc.},  {\bf 70} (1995),  581--604.

%\bibitem{H} Y. Heurteaux, Estimations de la dimension inf\'erieure
%et de la dimension sup\'erieure des mesures, {\it Ann. Inst. H. Poincar\'e Probab. Statist.}.
%{\bf 34} (1998), 309-338.


%\bibitem{Kae04} A. K\"{a}enm\"{a}ki. On natural invariant measures on generalised iterated function systems.
%{\it Ann. Acad. Sci. Fenn. Math.}, {\bf 29} (2004), 419--458.


\bibitem{KePe96} R. Kenyon and  Y. Peres, Measures of full dimension on affine-invariant sets. {\em  Ergodic Theory Dynam. Systems}  {\bf 16} (1996), 307--323.

\bibitem{Kin95} J. F. King, The singularity spectrum for general Sierpinski carpets. {\it Adv. Math.} {\bf 116} (1995), 1--8.


\bibitem{Mak40} P. T. Maker. The ergodic theorem for a sequence of functions. {\it Duke Math. J.} {\bf 6} (1940), 27--30.
%
%
%\bibitem{Kre72}W.  Krieger,  On unique ergodicity.
%{\it Proceedings of the Sixth Berkeley Symposium on Mathematical Statistics and Probability}
% (Univ. California, Berkeley, Calif., 1970/1971), Vol. II: Probability theory, pp. 327--346,
% Univ. California Press, Berkeley, Calif., 1972.

\bibitem{LaGa92} S. Lalley and D. Gatzouras, Hausdorff and Box Dimensions of
Certain Self-Affine Fractals, {\it Indiana Univ. Math. J.} {\bf 41}
(1992), 533--568.

\bibitem{LeWa77} F. Ledrappier and P.  Walters,
A relativised variational principle for continuous transformations.
{\it J. London Math. Soc.}   {\bf 16}  (1977),  568–-576.

\bibitem{Luz06}N.  Luzia,
A variational principle for the dimension for a class of non-conformal repellers.
{\it Ergodic Theory Dynam. Systems}  {\bf 26}  (2006),   821–-845.

\bibitem{Mat95}P. Mattila,   {\it Geometry of sets and measures in Euclidean
spaces}. Cambridge University Press, 1995.


\bibitem{McM84} C.  McMullen,  The Hausdorff dimension of general Sierpinski carpets.  {\it  Nagoya Math. J.} {\bf 96} (1984), 1--9.
%\bibitem{Mis76} M. Misiurewicz,  Topological conditional entropy. {\it Studia Math.} {\bf 55} (1976), 175--200.

\bibitem{Mis76} M. Misiurewicz,
A short proof of the variational principle for a
${\Bbb Z}_{+}^{\N}$ action on a compact space.
{\it Asterisque} {\bf 40} (1976), 147-187.

%\bibitem{Mum06} A. Mummert,
%The thermodynamic formalism for almost-additive sequences.
%{\it Discrete Contin. Dyn. Syst.}, {\bf 16} (2006), 435--454.

\bibitem{Ols98} L. Olsen, Self-affine multifractal Sierpinski sponges in $\mathbb{R}^d$.
{\it Pacific J. Math.} {\bf 183} (1998), 143--199.

\bibitem{Par81} W. Parry, {\it Topics in ergodic theory}. Cambridge Tracts in Mathematics, 75. Cambridge University Press, Cambridge-New York, 1981.

\bibitem{Pes97}Ya.  Pesin,  {\it Dimension theory in dynamical systems. Contemporary views and applications.} Chicago Lectures in Mathematics. University of Chicago Press, Chicago, IL, 1997.

\bibitem{PePi84} Ya.  Pesin and B. S. Pitskel', Topological pressure and the variational principle for non-
compact sets. {\it Functional Anal. Appl.} {\bf 18} (1984), 50--63.

\bibitem{Rue73}
D. Ruelle,
Statistical mechanics on a compact set with $Z\sp{v}$
action satisfying expansiveness and specification.
{\it Trans. Amer. Math. Soc.} {\bf 187} (1973), 237--251.


\bibitem{Rue78}D. Ruelle, {\it Thermodynamic formalism}. The mathematical structures of classical equilibrium statistical mechanics.
Encyclopedia of Mathematics and its Applications, 5. Addison-Wesley Publishing Co., Reading, Mass., 1978.

\bibitem{Rue82}
D. Ruelle,
Repellers for real analytic maps. {\em Ergodic Theory Dynam. Systems} {\bf 2}(1982), 99--107.

\bibitem{Wal75}P. Walters,
A variational principle for the pressure of continuous
transformations.
{\it Amer. J. Math.} {\bf 97} (1975),
937--971.

\bibitem{Wal82} P. Walters, {\it An introduction to ergodic theory},
         Springer-Verlag, Berlin, Heidelberg, New York, 1982.
\bibitem{Yay11} Y. Yayama,
Applications of a relative variational principle to dimensions of nonconformal expanding maps.
{\it Stoch. Dyn.} {\bf  11}  (2011),   643–-679.

%\bibitem{Y}  L.-S Young, Dimension, entropy and Lyapunov exponents.
%{\it Ergod. Th. \& Dynam. Sys.} {\bf 2}(1982), 109-124.

\bibitem{ZhCa08} Y. Zhao and Y. L. Cao, On the topological pressure of random
bundle transformations in sub-additive case.  {\it J. Math. Anal.
Appl.} {\bf 342} (2008), 715–-725.
\end{thebibliography}
\end{document}